\newcommand{\version}{paper}
\newcommand{\figuresize}{}

\newcommand{\habil}{habil}
\newcommand{\paper}{paper}
\ifx\version\habil
 \documentclass[a4paper,toc=bib]{scrbook} 
 \newcommand{\doc}{thesis\xspace}
\else
 \documentclass{scrartcl} 
\newcommand{\doc}{paper\xspace}
 \newenvironment{savequote}{\comment}{}
\fi

\usepackage[amsthm,greekit]{pamas}

\usepackage{verbatim,ifthen}
\usepackage{pgflibraryshapes}  
\usetikzlibrary{positioning}  
\usepackage{booktabs}  
\usepackage{accents}  
\usepackage{algorithm,algorithmic} 
\usepackage{enumitem}

\usepackage{eucal}

\DeclareMathAccent{\widehat}{\mathord}{largesymbols}{"62} 

\renewcommand{\todo}[1]{} 

\newcommand{\extra}[1]{\ifx\version\habil#1\fi}
\newcommand{\extrapaper}[1]{\ifx\version\paper#1\fi}

\extra{
\usepackage[avantgarde]{quotchap}
\let\subsubsection=\subsection
\let\subsection=\section
\let\section=\chapter
}
\extrapaper{
\let\charef=\secref
}

\newcommand{\f}{\mathcal{F}_0}
\newcommand{\fone}{\mathcal{F}}
\newcommand{\dom}{\overline{D}}
\newcommand{\cov}{\supset}
\newcommand{\mtc}{\mathit{MTC}}

\newcommand{\teq}[1][]{\ifthenelse{\equal{#1}{}}{\mathit{TEQ}}{\mathit{TEQ}(#1)}}
\newcommand{\me}[1][]{\ifthenelse{\equal{#1}{}}{\mathit{ME}}{\mathit{ME}(#1)}}
\newcommand{\mc}[1][]{\ifthenelse{\equal{#1}{}}{\mathit{MC}}{\mathit{MC}(#1)}}
\newcommand{\uc}[1][]{\ifthenelse{\equal{#1}{}}{\mathit{UC}}{\mathit{UC}(#1)}}
\newcommand{\ba}[1][]{\ifthenelse{\equal{#1}{}}{\mathit{BA}}{\mathit{BA}(#1)}}
\newcommand{\cnl}[1][]{\ifthenelse{\equal{#1}{}}{\mathit{CNL}}{\mathit{CNL}(#1)}}
\newcommand{\tc}[1][]{\ifthenelse{\equal{#1}{}}{\mathit{TC}}{\mathit{TC}(#1)}}
\newcommand{\co}[1][]{\ifthenelse{\equal{#1}{}}{\mathit{CO}}{\mathit{CO}(#1)}}
\newcommand{\bp}[1][]{\ifthenelse{\equal{#1}{}}{\mathit{BP}}{\mathit{BP}(#1)}}

\newcommand{\ms}[1]{\widehat{#1}}
\newcommand{\mcs}[1]{\accentset{\frown}{#1}}
\newcommand{\mr}[1]{\mathring{#1}}
\newcommand{\teqrel}[1][]{\ifthenelse{\equal{#1}{}}{\boldsymbol\rightarrow}{\teqrel_{#1}}}

\newcommand{\wsp}{\textsf{WSP}\xspace}
\newcommand{\ssp}{\textsf{SSP}\xspace}

\newcommand{\mon}{\textsf{MON}\xspace}
\newcommand{\iua}{\textsf{IUA}\xspace}
\newcommand{\com}{\textsf{COM}\xspace}
\newcommand{\irr}{\textsf{IRR}\xspace}

\newcommand{\aczero}{AC$^\text{0}$\xspace}
\newcommand{\tczero}{TC$^\text{0}$\xspace}

\ifx\version\habil
 \title{Tournament Solutions}
 \subtitle{Extensions of Maximality and their Applications to Decision-Making
\footnotetext{\centering \emph{Kumulative Habilitationsschrift an der Fakult\"at f\"ur Mathematik, Informatik und Statistik\\Ludwig-Maximilians-Universit\"at M\"unchen}}
  }
 \publishers{
\footnotesize
  \newcommand{\sw}{32em} 
  \newcommand{\sd}{3em} 
  \newcommand{\ra}{0.618} 
\begin{tikzpicture}
\draw (0,0) node[ellipse, draw, minimum width=\sw, minimum height=\ra*\sw] (CNL) {};
\draw (0,0) node[ellipse, draw, minimum width=\sw-\sd, minimum height=\ra*\sw-\sd] (TC) {};
\draw (0,0) node[ellipse, draw, minimum width=\sw-2*\sd, minimum height=\ra*\sw-2*\sd] (UC) {};
\draw (-\sd,0) node[ellipse, draw, minimum width=\sw-5*\sd, minimum height=\ra*\sw-3*\sd] (MC) {};
\draw (\sd,0) node[ellipse, draw, minimum width=\sw-5*\sd, minimum height=\ra*\sw-3*\sd] (BA) {};
\draw (0,0) node[ellipse, draw, minimum width=\sw-8*\sd, minimum height=\ra*\sw-4*\sd] (ME) {};
\draw (0,0) node[ellipse, draw, minimum width=\sw-9*\sd, minimum height=\ra*\sw-5*\sd] (TEQ) {$\teq$};

\draw[text=black] (0,-0.5*\ra*\sw) node[anchor=south](cnl){$\cnl$};
\draw[text=black] (0,-0.5*\ra*\sw+0.5*\sd) node[anchor=south](tc){$\tc$};
\draw[text=black] (0,-0.5*\ra*\sw+\sd) node[anchor=south](uc){$\uc$};
\draw[text=black] (-0.5*\sw+2*\sd,0) node[anchor=west](mc){$\mc$};
\draw[text=black] (0.5*\sw-2*\sd,0) node[anchor=east](ba){$\ba$};
\draw[text=black] (0,-0.5*\ra*\sw+2*\sd) node[anchor=south](me){$\me$};
\end{tikzpicture}
 
}

 \lowertitleback{
	\footnotesize\sffamily\itshape
	\begin{quote}
	Now the notion of domination on which we rely is, indeed, not transitive. [\dots] This lack of transitivity, especially in the above formalistic presentation, may appear to be an annoying complication and it may even seem desirable to make an effort to rid the theory of it.  Yet the reader who takes another look at the last paragraph will notice that it really contains only a circumlocution of a most typical phenomenon in all social organizations.  The domination relationships between various imputations $x$, $y$, $z$, \dots---i.e.~between various states of society---correspond to the various ways in which these can unstabilize---i.e.~upset---each other.  That various groups of participants acting as effective sets in various relations of this kind may bring about ‘cyclical’ dominations---\eg $y$ over $x$, $z$ over $y$, and $x$ over $z$---is indeed one of the most characteristic difficulties which a theory of these phenomena must face.

	Thus our task is to replace the notion of the optimum---i.e.~of the first element---by something which can take over its functions in a static equilibrium. This becomes necessary because the original concept has become untenable. We first observed its breakdown in the specific instance of a certain three-person game [\dots] But now we have acquired a deeper insight into the cause of its failure: it is the nature of our concept of domination, and specifically its intransitivity. 
	This type of relationship is not at all peculiar to our problem. Other instances of it are well known in many fields and it is regretted that they have never received a generic mathematical treatment. We mean all those concepts which are in the general nature of a comparison of preference or ‘superiority,’ or of order, but lack transitivity: \eg the strength of chess players in a tournament, the ‘paper form’ in sports and races, etc.\\[1.5ex]
	\flushright\upshape
	J.~von Neumann and O.~Morgenstern, 1944
	\end{quote}
 }
\else
 \title{Minimal Stable Sets in Tournaments}
\fi

\author{Felix Brandt\\
	  Technische Universit\"at M\"unchen\\85748 Garching bei M\"unchen, Germany \\
	  \texttt{\small brandtf@in.tum.de}
	}
\date{}

\begin{document}

\maketitle
\extra{\tableofcontents}

\extrapaper{
\begin{abstract}
We propose a systematic methodology for defining tournament solutions as extensions of maximality. The central concepts of this methodology are  maximal qualified subsets and minimal stable sets.  We thus obtain an infinite hierarchy of tournament solutions, encompassing the top cycle, the uncovered set, the Banks set, the minimal covering set, the tournament equilibrium set, the Copeland set, and the bipartisan set. Moreover, the hierarchy includes a new tournament solution, the \emph{minimal extending set}, which is conjectured to refine both the minimal covering set and the Banks set.
\end{abstract}
}

\begin{savequote}
\sffamily
One of the consequences of the assumptions of rational choice is that the choice in any environment can be determined by a knowledge of the choices in two-element environments.
\qauthor{K.~J.~Arrow, 1951}
\end{savequote}


\section{Introduction}\label{sec:intro}

Given a finite set of alternatives and choices between all pairs of alternatives, how to choose from the entire set in a way that is faithful to the pairwise comparisons? This simple, yet captivating, problem is studied in the literature on tournament solutions.
A tournament solution thus seeks to identify the ``best'' elements according to some binary dominance relation, which is usually assumed to be asymmetric and complete. Since the ordinary notion of maximality may return no elements due to cyclical dominations, numerous alternative solution concepts have been devised and axiomatized \citep[see, \eg][]{Moul86a,Lasl97a}. 
\extrapaper{
Tournament solutions have numerous applications throughout economics, most prominently in social choice theory where the dominance relation is typically defined via majority rule~\citep[\eg][]{Fish77a,Bord83a}. Other application areas include multi-criteria decision analysis~\citep[\eg][]{ArRa86a,BMP+06a}, non-cooperative game theory~\citep[\eg][]{FiRy95a,LLL93b,DuLe96a}, and cooperative game theory~\citep{Gill59a,BrHa09a}.
} \extra{The contribution of this thesis is threefold.
\begin{itemize}
\item A new methodology for defining tournament solutions based on the notions of \emph{maximal qualified subsets} and \emph{minimal stable sets} is proposed. This unifying framework clarifies the relationships between known tournament solutions, sheds more light on the enigmatic tournament solution $\teq$, and yields the definition of a new tournament solution called $\me$. These results are presented for the first time in this thesis (Chapters~\ref{sec:prelim} to \ref{sec:quant}).
\item The computational complexity of most tournament solutions is analyzed. While some tournament solutions can be computed very efficiently (\eg in linear time), others are shown to be NP-hard and thus computationally infeasible. These findings are of relevance to theoreticians who study and compare tournament solutions based on formal properties and to practitioners who wish to algorithmically implement tournament solutions. The results of this analysis were published by \citet{BrFi08b}, \citet{BFH09b}, and \cite{BFHM09a} and are summarized in Chapter~\ref{sec:computation}.
\item The applicability of tournament solutions to three elementary cases of decision-making is demonstrated. In particular, Chapter~\ref{sec:applications} addresses \emph{collective decision-making} (social choice theory), \emph{adversarial decision-making} (theory of zero-sum games), and \emph{coalitional decision-making} (cooperative game theory). \thmref{thm:ntu} in \secref{sec:coalitional} first appeared in a paper by \citet{BrHa09a}.
\end{itemize}

In the remainder of this chapter, we outline the techniques and results of the three parts listed above.

\subsection{Extensions of Maximality}

}


\extrapaper{In this paper, we}\extra{We} approach the tournament choice problem using a methodology consisting of two layers: \emph{qualified subsets} and \emph{stable sets}. 
Our framework captures most known tournament solutions (notable omissions are the Slater set and the Markov set) and allows us to provide unified proofs of properties and inclusion relationships between tournament solutions. 

In \secref{sec:prelim}, we introduce the terminology and notation required to handle tournaments and define six standard properties of tournament solutions: monotonicity ($\mon$), independence of unchosen alternatives ($\iua$), the weak superset property ($\wsp$), the strong superset property ($\ssp$),  composition-consistency ($\com$), and irregularity ($\irr$). The remainder of the paper is then divided into four sections.

\paragraph{Qualified Subsets \normalfont{(\secref{sec:maxsubsets})}}

The point of departure for our methodology is to collect the maximal elements of so-called qualified subsets, \ie distinguished subsets that admit a maximal element. In general, families of qualified subsets are characterized by three properties (closure, independence, and fusion). Examples of families of qualified subsets are all subsets with at most two elements, all subsets that admit a maximal element, or all transitive subsets.
  Each family yields a corresponding tournament solution and we thus obtain an infinite hierarchy of tournament solutions. The tournament solutions corresponding to the three examples given above are the set of \emph{all alternatives except the minimum}, the \emph{uncovered set} \citep{Fish77a,Mill80a}, and the \emph{Banks set} \citep{Bank85a}. Our methodology allows us to easily establish a number of inclusion relationships between tournament solutions defined via qualified subsets (\propref{pro:inclusions}) and to prove that all such tournament solutions satisfy $\wsp$ and $\mon$ (\propref{pro:qwspmon}).
Based on an axiomatic characterization using minimality and a new property called strong retentiveness, we show that the Banks set is the finest tournament solution definable via qualified subsets (\thmref{thm:bankssdc}).

\paragraph{Stable Sets \normalfont{(\secref{sec:minstablesets})}}

Generalizing an idea by \citet{Dutt88a}, we then propose a method for refining any suitable solution concept $S$ by defining minimal sets that satisfy a natural stability criterion with respect to $S$. A crucial property in this context is whether $S$ always admits a \emph{unique} minimal stable set. For tournament solutions defined via qualified subsets, we show that this is the case if and only if no tournament contains two disjoint stable sets (\lemref{lem:pairwise}). As a consequence of this characterization and a theorem by \citet{Dutt88a}, we show that an infinite number of tournament solutions (defined via qualified subsets) always admit a unique minimal stable set (\thmref{thm:uniquekcovering}). Moreover, we show that all tournament solutions defined as unique minimal stable sets satisfy $\wsp$ and $\iua$ (\propref{pro:stableaiz}), $\ssp$ and various other desirable properties if the original tournament solution is defined via qualified subsets (\thmref{thm:minstableproperties}), and $\mon$ and $\com$ if the original tournament solution satisfies these properties (\propref{pro:stablemon} and \propref{pro:stablecomp}).
The minimal stable sets with respect to the three tournament solutions mentioned in the paragraph above are the \emph{minimal dominant set}, better known as the \emph{top cycle} \citep{Good71a,Smit73a}, the \emph{minimal covering set} \citep{Dutt88a}, and a new tournament solution that we call the \emph{minimal extending set ($\me$)}. Whether $\me$ satisfies uniqueness turns out to be a highly non-trivial combinatorial problem and remains open. If true, $\me$ would be contained in both the minimal covering set and the Banks set while satisfying a number of desirable properties. 
We conclude the section by axiomatically characterizing all tournament solutions definable via unique minimal stable sets (\propref{pro:stablecharac}).

\paragraph{Retentiveness and Stability \normalfont{(\secref{sec:teq})}}
$\me$ bears some resemblance to Schwartz's \emph{tournament equilibrium set} $\teq$ \citep{Schw90a}, which is defined as the minimal \emph{retentive} set of a tournament. There are some similarities between retentiveness and stability and, as in the case of $\me$, the uniqueness of a minimal retentive set and thus the characteristics of $\teq$ remain an open problem \citep{LLL93a,Houy09a}. 
We show that Schwartz's conjecture is stronger than ours (\thmref{thm:teqsq}) and has a number of interesting consequences such as that $\teq$ can also be represented as a minimal stable set (\thmref{thm:teqstable}) and is strictly contained in $\me$ (\coref{cor:teqsubset}). 

\paragraph{Quantitative Concepts \normalfont{(\secref{sec:quant})}}
We also briefly discuss a quantitative version of our framework, which considers qualified subsets that are maximal in terms of cardinality rather than set inclusion. We thus obtain the \emph{Copeland set} \citep{Cope51a} and---using a slightly modified definition of stability---the \emph{bipartisan set} \citep{LLL93b}. This completes our picture of tournament solutions and their corresponding minimal stable sets as given in \tabref{tbl:stronglystable} on page~\pageref{tbl:stronglystable}.

\extra{
\subsection{Computational Aspects}

Since tournament solutions are usually studied in economics, computational considerations are often ignored or not addressed in a rigorous and formal way. Nevertheless, the effort required to compute a tournament solution is of great importance for if computing a tournament solution is intractable, its applicability is seriously undermined. For all major tournament solutions defined in the first part of the thesis, we address the following questions: 
\begin{itemize}
\item Is there a polynomial-time algorithm for computing the tournament solution? 
\item If so, can it even be computed in linear time and/or is it first-order definable (and thus contained in the complexity class \aczero)? 
\item If not, can we show NP-hardness of the membership decision problem? 
\end{itemize}

Unsurprisingly, there is a strong correlation between the number of attractive properties a tournament solution satisfies and its computational complexity. The minimal covering set and the bipartisan set are particularly noteworthy in this context as they satisfy a large number of attractive properties, yet are barely computable in polynomial time. 

\subsection{Applications to Decision-Making}

Tournament solutions have numerous applications in diverse areas such as \emph{sports competitions} \citep[see, \eg][]{Usha76a,MGGM05a}, \emph{webpage and journal ranking} \citep[see, \eg][]{KoSt07a,BrFi07b}, \emph{biology} \citep[see, \eg][]{Land51a,Land51b,Land53a}, and \emph{psychology} \citep[see, \eg][]{Schj22a,Slat61a}. In this thesis, we focus on applications related to fundamental problems in decision-making, namely 
\begin{itemize}
\item \emph{collective decision-making} (social choice theory), 
\item \emph{adversarial decision-making} (theory of zero-sum games), and
\item \emph{coalitional decision-making} (cooperative game theory). 
\end{itemize}
In social choice theory, we characterize a setting where social choice functions and tournament solutions coincide. Indeed, most well-known tournament solutions were introduced as social choice functions based on the pairwise majority relation. We then establish a similar equivalence for certain game-theoretic solution concepts in a subclass of two-player zero-sum games. Interestingly, many tournament solutions correspond to well-known game-theoretic solution concepts within this particular class of games. Finally, we discuss how tournament solutions may be applied as solution concepts for cooperative game theory. This connection is not as elaborate as the previous two as it requires the generalization of tournament solutions to directed graphs in order to obtain a meaningful equivalence.

}

\begin{savequote}
\sffamily
Tournaments form perhaps the most interesting class of directed graphs and have a very rich theory, a theory which has no analog in the theory of undirected graphs.
\qauthor{K.~B.~Reid and L.~W.~Beineke, 1978}
\end{savequote}

\section{Preliminaries}\label{sec:prelim}

The core of the problem studied in the literature on tournament solutions is how to extend choices in sets consisting of only two elements to larger sets. Thus, our primary objects of study will be functions that select one alternative from any pair of alternatives. Any such function can be conveniently represented by a tournament, \ie a binary relation on the entire set of alternatives. Tournament solutions then advocate different views on how to choose from arbitrary subsets of alternatives based on these pairwise comparisons \citep[see, \eg][for an excellent overview of tournament solutions]{Lasl97a}.

\subsection{Tournaments}

Let $X$ be a universe of alternatives. The set of all \emph{finite} subsets of set $X$ will be denoted by $\f(X)$ whereas the set of all \emph{non-empty} finite subsets of $X$ will be denoted by $\fone(X)$.
A (finite) \emph{tournament $T$} is a pair $(A,{\succ})$, where $A\in \fone(X)$ and $\succ$ is an asymmetric and complete (and thus irreflexive) binary relation on $X$, usually referred to as the \emph{dominance relation}.\footnote{This definition slightly diverges from the common graph-theoretic definition where $\succ$ is defined on $A$ rather than $X$. However, it facilitates the sound definition of tournament functions (such as tournament solutions or concepts of qualified subsets).} Intuitively, $a\succ b$ signifies that alternative $a$ is preferable to $b$. 
The dominance relation can be extended to sets of alternatives by writing $A\succ B$ when $a\succ b$ for all $a\in A$ and $b\in B$. When $A$ or $B$ are singletons, we omit curly braces to improve readability.
We further write $\mathcal{T}(X)$ for the set of all tournaments on $X$.

\extra{For a relation $R$, we define $R^k$ recursively by letting $R^1=R$ and 
$R^{k+1}=R^k\cup\{(a,b)\mid (a,c)\in R^k \text{ and } (c,b)\in R \text{ for some }c\}$. The transitive closure of $R$ is defined as $R^*=\bigcup_{k\in \mathbb{N}} R^k$.}
For a set $B$, a relation $R$, and an element $a$, 
we denote by $D_{B,R}(a)$ the \emph{dominion} of $a$ in $B$, \ie 
\[D_{B,R}(a)=\{\,b\in B\mid a\mathrel{R}b\}\text{,}\] 
and by $\dom_{B,R}(a)$ the \emph{dominators} of $a$ in~$B$, \ie 
\[\dom_{B,R}(a)=\{\,b\in B\mid b\mathrel{R}a\}\text{.}\]
Whenever the tournament $(A,\succ)$ is known from the context and $R$ is the dominance relation $\succ$ or $B$ is the set of all alternatives $A$, the respective subscript will be omitted to improve readability. 

The \emph{order} of a tournament $T=(A,{\succ})$ refers to the cardinality of $A$.
A tournament $T=(A,{\succ})$ is called \emph{regular} if $|D(a)|=|D(b)|$ for all $a,b \in A$. 

Let $T=(A,{\succ})$ and $T'=(A',\succ')$ be two tournaments. A \emph{tournament isomorphism} of $T$ and $T'$ is a bijective mapping $\pi:A\rightarrow A'$ such that $a\succ b$ if and only if $\pi(a)\succ' \pi(b)$. 
A tournament $T=(A,{\succ})$ is called \emph{homogeneous} (or vertex-transitive) if for every pair of alternatives $a,b\in A$ there is a tournament isomorphism $\pi:A\rightarrow A$ of $T$ such that $\pi(a)=b$.

\subsection{Components and Decompositions}
\label{sec:components}

An important structural concept in the context of tournaments is that of a \emph{component}. A component is a subset of alternatives that bear the same relationship to all alternatives not in the set.

Let $T=(A,{\succ})$ be a tournament. A non-empty subset $B$ of $A$ is a \emph{component} of $T$ if for all $a\in A\setminus B$ either $B\succ a$ or $a\succ B$.
A \emph{decomposition} of $T$ is a set of pairwise disjoint components $\{B_1,\dots,B_k\}$ of $T$ such that $A=\bigcup_{i=1}^k B_i$.
\extra{The \emph{null decomposition} of a tournament $T=(A,{\succ})$ is $\{A\}$; the \emph{trivial decomposition} consists of all singletons of $A$. Any other decomposition is called \emph{proper}. A tournament is said to be \emph{decomposable} if it admits a proper decomposition.}
Given a particular decomposition $\tilde{B}=\{B_1,\dots,B_k\}$ of $T$, the \emph{summary} of $T$ is defined as the tournament on the individual components rather than the alternatives. Formally, the summary
$\tilde{T}=(\tilde{B},\tilde{\succ})$ of $T$ is the tournament such that for all $i,j\in\{1,\dots,k\}$ with $i\ne j$, 
\[
B_i\mathrel{\tilde{\succ}}B_j \quad\text{ if and only if }\quad B_i\succ B_j\text{.}
\]
Conversely, a new tournament can be constructed by replacing each alternative with a component. For notational convenience, we tacitly assume that $\mathbb{N}\subseteq X$. 
For pairwise disjoint sets $B_1,\dots,B_k\subseteq X$ and tournaments $\tilde{T}=(\{1,\dots,k\},\tilde{\succ})$, $T_1=(B_1,\succ_1)$, \dots, $T_k=(B_k,\succ_k)$, the \emph{product} of $T_1,\dots,T_k$ with respect to $\tilde{T}$, denoted by $\Pi(\tilde{T},T_1,\dots,T_k)$, is a tournament $(A,{\succ})$ such that $A=\bigcup_{i=1}^kB_i$ and for all $b_1\in B_i, b_2\in B_j$, 
\[
b_1\succ b_2 \quad\text{ if and only if }\quad i=j \text{ and } b_1\succ_i b_2 \text{, or } i\ne j \text{ and } i\mathrel{\tilde{\succ}} j\text{.}
\]

Components can also be used to simplify the graphical representation of tournaments. We will denote components by gray circles. Furthermore, omitted edges in figures that depict tournaments are assumed to point downwards or from left to right by convention (see, \eg \figref{fig:mcme}\extra{ on page~\pageref{fig:mcme}).}\extrapaper{).}

\subsection{Tournament Functions}

A central aspect of this \doc will be functions that, for a given tournament, yield one or more subsets of alternatives. We will therefore define the notion of a \emph{tournament function}.
A function on tournaments is a tournament function if it is independent of outside alternatives and stable with respect to tournament isomorphisms.
A tournament function may yield a (non-empty) subset of alternatives---as in the case of tournament solutions---or a set of subsets of alternatives---as in the case of qualified or stable sets.
\begin{definition}\label{def:tfunction}
Let $Z\in \{\f(X),\fone(X),\fone(\fone(X))\}$. A function $f:\mathcal{T}(X)\rightarrow Z$ is a \emph{tournament function} if
\begin{enumerate}[label=\textit{(\roman*)}] 
\item $f(T)=f(T')$ for all tournaments $T=(A,{\succ})$ and $T'=(A,\succ')$ such that ${\succ}|_A={\succ'}|_A$, and
\item $f((\pi(A),{\succ'}))=\pi(f((A,{\succ})))$ for all tournaments $(A,{\succ})$,  $(A',\succ')$, and tournament isomorphisms\footnote{$\pi(A)$ is a shorthand for the set $\{\pi(a)\mid a \in A\}$.}  $\pi:A\rightarrow A'$ of $(A,{\succ})$ and $(A',{\succ'})$.
\end{enumerate}
\end{definition}
For a given set $B\in\fone(X)$ and tournament function $f$, we overload $f$ by also writing~$f(B)$, provided the dominance relation is known from the context. For two tournament functions $f$ and $f'$, we write $f'\subseteq f$ if $f'(T)\subseteq f(T)$ for all tournaments~$T$.

\subsection{Tournament Solutions}

The first tournament function we consider is $\max_\prec:\mathcal{T}(X)\rightarrow \f(X)$, which returns the undominated alternatives of a tournament. Formally, 
\[
\max_\prec((A,{\succ})) = \{ a\in A \mid \dom_\succ(a)=\emptyset\}\text{.}
\]
Due to the asymmetry of the dominance relation, this function returns at most one alternative in any tournament. Moreover, maximal---\ie undominated---and maximum---\ie dominant---elements coincide.
In social choice theory, the maximum of a majority tournament is commonly referred to as the \emph{Condorcet winner}. Obviously, the dominance relation may contain cycles and thus fail to have a maximal element. For this reason, a variety of alternative concepts to single out the ``best'' alternatives of a tournament have been suggested. Formally, a \emph{tournament solution $S$} is defined as a function that associates with each tournament $T=(A,{\succ})$ a non-empty subset~$S(T)$ of $A$. 
\begin{definition}\label{def:tsolution}
A \emph{tournament solution $S$} is a tournament function $S:\mathcal{T}(X)\rightarrow \fone(X)$ such that
$\displaystyle \max_\prec(T)\subseteq S(T)\subseteq A$ for all tournaments $T=(A,{\succ})$.\footnote{\citet{Lasl97a} is slightly more stringent here as he requires the maximum be the only element in $S(T)$ whenever it exists.}
\end{definition}
The set $S(T)$ returned by a tournament solution for a given tournament $T$ is called the \emph{choice set} of $T$ whereas $A\setminus S(T)$ consists of the \emph{unchosen alternatives}. 
Since tournament solutions always yield non-empty choice sets, they have to select all alternatives in homogeneous tournaments.
If $S'\subseteq S$ for two tournament solutions $S$ and $S'$, we say that $S'$ is a \emph{refinement} of $S$ or that $S'$ is \emph{finer} than~$S$.

\subsection{Properties of Tournament Solutions}
\label{sec:properties}

The literature on tournament solutions has identified a number of desirable properties for tournament solutions. In this section, we will define six of the most common properties.\footnote{Our terminology slightly differs from the one by \citet{Lasl97a} and others. \emph{Independence of unchosen alternatives} is also called \emph{independence of the losers} or \emph{independence of non-winners}. The \emph{weak superset property} has been referred to as $\epsilon^+$ or the \emph{A\"izerman property}.} In a more general context, \citet{Moul88a} distinguishes between \emph{monotonicity} and \emph{independence} conditions, where a monotonicity condition describes the positive association of the solution with some parameter and an independence condition characterizes the invariance of the solution under the modification of some parameter.

In the context of tournament solutions, we will further distinguish between properties that are defined in terms of the dominance relation and properties defined in terms of the set inclusion relation. With respect to the former, we consider \emph{monotonicity} and \emph{independence of unchosen alternatives}.
A tournament solution is monotonic if a chosen alternative remains in the choice set when extending its dominion and leaving everything else unchanged. 
\begin{definition}\label{def:mon}
  A tournament solution~$S$ satisfies \emph{monotonicity} (\mon) if $a\in S(T)$ implies $a\in S(T')$ for all tournaments $T=(A,{\succ})$, $T'=(A,\succ')$, and $a\in A$ such that ${\succ}|_{A\setminus \{a\}}={\succ'}|_{A\setminus \{a\}}$ and $D_{\succ}(a)\subseteq D_{\succ'}(a)$.
\end{definition}
A solution is independent of unchosen alternatives if the choice set is invariant under any modification of the dominance relation between unchosen alternatives.
\begin{definition}\label{def:ira}
  A tournament solution~$S$ is \emph{independent of unchosen alternatives} (\iua) if $S(T)=S(T')$ for all tournaments $T=(A,{\succ})$ and $T'=(A,\succ')$ such that $D_{\succ}(a)=D_{\succ'}(a)$ for all $a\in S(T)$.
\end{definition}

With respect to set inclusion, we consider a monotonicity property to be called the \emph{weak superset property} and an independence property known as the \emph{strong superset property}. A tournament solution satisfies the weak superset property if an unchosen alternative remains unchosen when other unchosen alternatives are removed.
\begin{definition}\label{def:wsp}
A tournament solution~$S$ satisfies the \emph{weak superset property} (\wsp) if $S(B)\subseteq S(A)$ for all tournaments $T=(A,{\succ})$ and $B\subseteq A$ such that $S(A)\subseteq B$.
\end{definition}
The strong superset property states that a choice set is invariant under the removal of unchosen alternatives.
\begin{definition}\label{def:ssp}
A tournament solution~$S$ satisfies the \emph{strong superset property} (\ssp) if $S(B)=S(A)$ for all tournaments $T=(A,{\succ})$ and $B\subseteq A$ such that $S(A)\subseteq B$.
\end{definition}
The difference between \wsp and \ssp is precisely another independence condition called \emph{idempotency}. A solution is \emph{idempotent} 
if the choice set is invariant under repeated application of the solution concept, \ie $S(S(T))=S(T)$ for all tournaments $T$. 
When $S$ is not idempotent, we define $S^k(T)=S(S^{k-1}(T))$ inductively by letting $S^1(T)=S(T)$ and $S^\infty(T)=\bigcap_{k\in\mathbb{N}} S^k(T)$.

The four properties defined above (\mon, \iua, \wsp, and \ssp) will be called \emph{basic} properties of tournament solutions. The conjunction of \mon and \ssp implies \iua. It is therefore sufficient to show \mon and \ssp in order to prove that a tournament solution satisfies all four basic properties.

Two further properties considered in this \doc are \emph{composition-consistency} and \emph{irregularity}. A tournament solution is composition-consistent if it chooses the ``best'' alternatives from the ``best'' components.
\begin{definition}
A tournament solution $S$ is \emph{composition-consistent} (\com) if for all tournaments $T$, $T_1, \dots, T_k$, and $\tilde{T}$ such that $T=\Pi(\tilde{T},T_1,\dots,T_k)$, $S(T)=\bigcup_{i\in S(\tilde{T})} S(T_i)$.
\end{definition}
Finally, a tournament solution is irregular if it is capable of excluding alternatives in regular tournaments.
\begin{definition}
A tournament solution $S$ satisfies \emph{irregularity} (\irr) if there exists a regular tournament $T=(A,\succ)$ such that $S(T)\ne A$.
\end{definition}


\begin{savequote}
\sffamily
The social sciences may be characterized by the fact that in most of their problems numerical measurements seem to be absent and considerations of space are irrelevant.
\qauthor{J.~G.~Kemeny, 1959}
\end{savequote}

\section{Qualified Subsets}
\label{sec:maxsubsets}

In this \extra{chapter}\extrapaper{section}, we will define a class of tournament solutions that is based on identifying significant subtournaments of the original tournament, such as subtournaments that admit a maximal alternative. 

\subsection{Concepts of Qualified Subsets}

A \emph{concept of qualified subsets} is a tournament function that, for a given tournament $T=(A,{\succ})$, returns subsets of $A$ that satisfy certain properties. Each such set of sets will be referred to as a \emph{family of qualified subsets}.
Two natural examples of concepts of qualified subsets are $\mathcal{M}$, which yields all subsets that admit a maximal element, and $\mathcal{M}^*$, which yields all non-empty transitive subsets. Formally,
\begin{eqnarray*}
\mathcal{M}((A,{\succ})) &=& \{B\subseteq A\mid \max_\prec(B)\ne\emptyset\}\text{ and}\\
\mathcal{M}^*((A,{\succ})) &=& \{B\subseteq A\mid \max_\prec(C)\ne\emptyset \text{ for all non-empty } C\subseteq B\}\text{.}
\end{eqnarray*}

$\mathcal{M}$ and $\mathcal{M}^*$ 
are examples of concepts of qualified subsets, which are formally defined as follows.


\begin{definition}\label{def:qualified}
Let $\mathcal{Q}:\mathcal{T}(X)\rightarrow \fone(\fone(X))$ be a tournament function such that $\mathcal{M}_1(T)\subseteq\mathcal{Q}(T)\subseteq \mathcal{M}(T)$. $\mathcal{Q}$ is a \emph{concept of qualified subsets} if it meets the following three conditions for every tournament $T=(A,{\succ})$.
\begin{description}[font=\normalfont]
 \item[(Closure)] $\mathcal{Q}(T)$ is downward closed with respect to $\mathcal{M}$: Let $Q\in  \mathcal{Q}(T)$. Then, $Q'\in \mathcal{Q}(T)$ if $Q'\subseteq Q$ and $Q'\in \mathcal{M}(T)$. 
 \item[(Independence)] 
Qualified sets are independent of outside alternatives: 
 Let $A'\in \fone(X)$ and $Q\subseteq A\cap A'$. Then, $Q\in \mathcal{Q}(A)$ if and only if $Q\in \mathcal{Q}(A')$.
 \item[(Fusion)] Qualified sets may be merged under certain conditions:
Let $Q_1, Q_2\in \mathcal{Q}(T)$ and $Q_1\setminus Q_2\succ Q_2$. Then $Q_1\cup Q_2\in \mathcal{Q}(T)$ if there is a tournament $T'\in \mathcal{T}(X)$ and $Q\in\mathcal{Q}(T')$ such that $|Q_1\cup Q_2|\le|Q|$.
\end{description}
\end{definition}
Whether a set is qualified only depends on its internal structure (due to independence and the isomorphism condition of \defref{def:tfunction}).
While closure, independence, and the fact that all singletons are qualified are fairly natural, the fusion condition is slightly more technical. Essentially, it states that if a qualified subset dominates another qualified subset, then the union of these subsets is also qualified. The additional cardinality restriction is only required to enable \emph{bounded} qualified subsets. For every concept of qualified subsets $\mathcal{Q}$ and every given $k\in\mathbb{N}$, $\mathcal{Q}_k:\mathcal{T}(X)\rightarrow \fone(\fone(X))$ is a tournament function such that $\mathcal{Q}_k(T)=\{B\in \mathcal{Q}(T)\mid |B|\le k\}$.
It is easily verified that $\mathcal{Q}_k$ is a concept of qualified subsets.
Furthermore, $\mathcal{M}$ and $\mathcal{M}^*$ (and thus also $\mathcal{M}_k$ and $\mathcal{M}^*_k$) are concepts of qualified subsets. Since only tournaments of order $4$ or more may be intransitive and admit a maximal element at the same time, $\mathcal{M}_k=\mathcal{M}_k^*$ for $k\in\{1,2,3\}$.

\subsection{Maximal Elements of Maximal Qualified Subsets}
\label{sec:maxmaxqualified}

For every concept of qualified subsets, we can now define a tournament solution that yields the maximal elements of all inclusion-maximal qualified subsets, \ie all qualified subsets that are not contained in another qualified subset.
\begin{definition}\label{def:maxqualified}
Let $\mathcal{Q}$ be a concept of qualified subsets. Then, the tournament solution $S_\mathcal{Q}$ is defined as
\[
S_\mathcal{Q}(T) = \{ \max_\prec(B) \mid B\in \max_\subseteq(\mathcal{Q}(T))\}\text{.}
\]
\end{definition}

Since any family of qualified subsets contains all singletons, $S_\mathcal{Q}(T)$ is guaranteed to be non-empty and contains the Condorcet winner whenever one exists. As a consequence, $S_\mathcal{Q}$ is well-defined as a tournament solution. 

The following  tournament solutions can be restated via appropriate concepts of qualified subsets.

\paragraph{Condorcet non-losers.}
$S_{\mathcal{M}_2}$ is arguably the largest non-trivial tournament solution. In tournaments of order two or more, it chooses every alternative that dominates at least one other alternative. We will refer to this concept as \emph{Condorcet non-losers} ($\cnl$) as it selects everything except the minimum (or \emph{Condorcet loser}) in such tournaments.
 
\paragraph{Uncovered set.}
$S_\mathcal{M}(T)$ returns the \emph{uncovered set} $\uc(T)$ of a tournament $T$, \ie the set consisting of the maximal elements of inclusion-maximal subsets that admit a maximal element. The uncovered set is usually defined in terms of a subrelation of the dominance relation called the \emph{covering relation} \citep{Fish77a,Mill80a}. \extra{See also \secref{sec:uc_complexity}.}

\paragraph{Banks set.}
$S_{\mathcal{M}^*}(T)$ yields the \emph{Banks set} $\ba(T)$ of a tournament $T$ \citep{Bank85a}. $\mathcal{M}^*(T)$ contains subsets that not only admit a maximum, but can be completely ordered from maximum to minimum such that all of their non-empty subsets admit a maximum. $S_{\mathcal{M}^*}(T)$ thus returns the maximal elements of inclusion-maximal \emph{transitive} subsets.
\bigskip

In the remainder of this section, we will prove various statements about tournaments solutions defined via qualified subsets. 
For a set $B$ and an alternative $a\not\in B$, the short notation $[B,a]$ will be used to denote the set $B\cup\{a\}$ and the fact that $\max_\prec(B\cup\{a\})=\{a\}$.
In several proofs, we will make use of the fact that whenever $a\not\in S_\mathcal{Q}(T)$, there is some $b\in S_\mathcal{Q}(T)$ for every qualified subset $[Q,a]$ such that $[Q\cup\{a\},b]\in \mathcal{Q}(T)$.
We start by showing that every tournament solution defined via qualified subsets satisfies the weak superset property and monotonicity. 

\begin{proposition}\label{pro:qwspmon}
Let $\mathcal{Q}$ be a concept of qualified subsets. Then, $S_\mathcal{Q}$ satisfies \wsp and \mon.
\end{proposition}
\begin{proof}
Let $T=(A,\succ)$ be a tournament, $a\not\in S_\mathcal{Q}(A)$, and $A'\subseteq A$ such that $S_\mathcal{Q}(T)\cup\{a\}\subseteq A'$. For \wsp, we need to show that $a\not\in S_\mathcal{Q}(A')$.
Let $[Q,a]\in\mathcal{Q}(A')$. Due to independence, $[Q,a]\in\mathcal{Q}(A)$. Since $a\not\in S_\mathcal{Q}(A)$, there has to be some $b\in S_\mathcal{Q}(A)$ such that $[Q\cup\{a\},b]\in\mathcal{Q}(A)$. Again, independence implies that $[Q\cup\{a\},b]\in\mathcal{Q}(A')$. Hence, $a\not\in S_\mathcal{Q}(A')$.

For \mon, observe that $a\in S_\mathcal{Q}$ implies that there exists $[Q,a]\in \max_\subseteq(\mathcal{Q}(T))$. 
Define $T'=(A,\succ')$ by letting $T'|_{A\setminus \{a\}}=T|_{A\setminus \{a\}}$ and $a\succ' b$ for some $b\in A$ with $b\succ a$.
Clearly, $[Q,a]$ is contained in $\mathcal{Q}(T')$ due to independence and the fact that $b\not\in Q$. Now, assume for contradiction that there is some $c\in A$ such that $[Q\cup\{a\},c]\in\mathcal{Q}(T')$. Since $a\succ' b$, $c\ne b$. Independence then implies that $[Q\cup\{a\},c]\in\mathcal{Q}(T)$, a contradiction.
\end{proof}
Proposition~\ref{pro:qwspmon} implies several known statements such as that
$\cnl$, $\uc$, and $\ba$ satisfy \mon and \wsp. All three concepts are known to fail idempotency (and thus \ssp). $\cnl$ trivially satisfies \iua whereas this is not the case for $\uc$ and $\ba$ \citep[see][]{Lasl97a}.
We also obtain some straightforward inclusion relationships, which define an infinite hierarchy of tournament solutions ranging from $\cnl$ to $\ba$.

\begin{proposition}\label{pro:inclusions}
  $S_{\mathcal{M}^*}\subseteq S_{\mathcal{M}}$, $S_{\mathcal{M}_k^*}\subseteq S_{\mathcal{M}_k}$, $S_{\mathcal{Q}_{k+1}}\subseteq S_{\mathcal{Q}_k}$, and $S_{\mathcal{Q}}\subseteq S_{\mathcal{Q}_{k}}$ for every concept of qualified subsets $\mathcal{Q}$ and $k\in\mathbb{N}$.
\end{proposition}
\begin{proof}
All inclusion relationships follow from the following observation.
Let $T$ be a tournament and $\mathcal{Q}$ and $\mathcal{Q'}$ concepts of qualified subsets such that for every $[Q,a]\in\max_\subseteq(\mathcal{Q}(T))$, there is $[Q',a]\in\max_\subseteq(\mathcal{Q'}(T))$. Then, $S_{\mathcal{Q}}\subseteq S_\mathcal{Q'}$.
\end{proof}

It turns out that the Banks set is the finest tournament solution definable via qualified subsets. In order to show this, we introduce a new property called \emph{strong retentiveness} that prescribes that the choice set of every dominator set is contained in the original choice set. Alternatively, it can be seen as a variant of \wsp because it states that a choice set may not grow when an alternative and its entire dominion are removed from the tournament.

\begin{definition}\label{def:dcon}
A tournament solution~$S$ satisfies \emph{strong retentiveness} if $S(\dom(a))\subseteq S(A)$ for all tournaments $T=(A,{\succ})$ and $a\in A$.
\end{definition}

\begin{lemma}\label{lem:dcon}
Let $\mathcal{Q}$ be a concept of qualified subsets. Then, $S_\mathcal{Q}$ satisfies strong retentiveness.
\end{lemma}
\begin{proof}
Let $(A,\succ)$ be a tournament, $a\in A$ an alternative, and $B=\dom(a)$. We show that $b\in S_\mathcal{Q}(B)$ implies that $b\in S_\mathcal{Q}(A)$.
Let $[Q,b]$ be a maximal qualified subset in $B$, \ie $[Q,b]\in\max_\subseteq(\mathcal{Q}(B))$. If $[Q,b]\in\max_\subseteq(\mathcal{Q}(A))$, we are done. Otherwise, there has to be some $c\in A$ such that $[Q\cup\{b\},c]\in \mathcal{Q}(A)$. Furthermore, $[Q,b]\succ a$ and $a\succ c$ because otherwise $[Q\cup\{b\},c]$ would be qualified in $B$ as well. We can now merge the qualified subsets $[Q,b]$ and $[\{a\},b]$ according to the fusion condition. We claim that $[Q\cup\{a\},b]\in\max_\subseteq(\mathcal{Q}(A))$. Assume for contradiction that there is some $d\in A$ such that $[Q\cup\{a,b\},d]\in \mathcal{Q}(A)$. Since $d\in \dom(a)$, independence implies that $[Q\cup\{a,b\},d]\in \mathcal{Q}(B)$. This is a contradiction because $[Q,b]$ was assumed to be a \emph{maximal} qualified subset of~$B$.
\end{proof}

\begin{theorem}\label{thm:bankssdc}
The Banks set is the finest tournament solution satisfying strong retentiveness and thus the finest tournament solution definable via qualified subsets.
\end{theorem}
\begin{proof}
Let $S$ be a tournament solution that satisfies strong retentiveness and $T=(A,\succ)$ a tournament. We first show that $\ba(A)\subseteq S(A)$. For every $a\in \ba(A)$, there has to be maximal transitive set $[Q,a]\subseteq A$. Let $Q=\{q_1,q_2,\dots,q_n\}$. We show that $B=\bigcap_{i=1}^n \dom(q_i)=\{a\}$. Since $a\succ Q$, $a\in B$. Assume for contradiction that $b\in B$ with $b\ne a$. Then $b\succ Q$ and either $[Q\cup\{b\},a]$ or $[Q\cup\{a\},b]$ is a transitive set, which contradicts the maximality of $[Q,a]$. Since $S(B)=S(\{a\})=\{a\}$, the repeated application of strong retentiveness implies that $a\in S(A)$.
The statement now follows from Lemma~\ref{lem:dcon}.
\end{proof}


\begin{savequote}
\sffamily
Thus our solutions $S$ correspond to such ``standards of behavior'' that have an inner stability: once they are generally accepted they overrule everything else and no part of them can be overruled within the limits of the accepted standards.
\qauthor{J.~v. Neumann and O.~Morgenstern, 1944}
\end{savequote}

\section{Stable Sets}
\label{sec:minstablesets}

In this \extra{chapter}\extrapaper{section}, we propose a general method for refining any suitable solution concept~$S$ by formalizing the stability of sets of alternatives with respect to $S$.  This method is based on the notion of stable sets \citep{vNM44a} and generalizes covering sets as introduced by \citet{Dutt88a}. 

\subsection{Stability and Directedness}\label{sec:directedness}

The reason why we are interested in maximal---\ie undominated---alternatives is that dominated alternatives can be upset by other alternatives; they are unstable. The rationale behind stable sets is that this instability is only meaningful if an alternative is upset by something which itself is stable. Hence, a set of alternatives $B$ is said to be stable if it consists precisely of those alternatives not upset by $B$. In von Neumann and Morgenstern's original definition, $a$ is upset by $B$ if some element of $B$ dominates $a$. In our generalization, $a$ is upset by $B$ if $a\not\in S(B\cup \{a\})$ for some underlying solution concept $S$.\footnote{Von Neumann and Morgenstern's definition can be seen as the special case where $a$ is upset by $B$ if $a\not\in \max_\prec(B\cup\{a\})$.}

As an alternative to this fixed-point definition, which will be formalized in \coref{cor:fpdef}, stable sets can be seen as sets that comply with internal and external stability in some well-defined way. 
First, there should be no reason to restrict the selection by excluding some alternative from it and, secondly, there should be an argument against each proposal to include an outside alternative into the selection.\footnote{A large number of solution concepts in the social sciences spring from similar notions of internal and/or external stability \citep[see, \eg][]{vNM44a,Nash51a,Shap64a,Schw86a,Dutt88a,BaWe91a,DuLe96a}. 
\citet{Wils70a} refers to stability as the \emph{solution property}.} In our context, external stability with respect to some tournament solution $S$ is defined as follows.

\begin{definition}\label{def:sstable}
Let $S$ be a tournament solution and $T=(A,{\succ})$ a tournament.  Then, $B\subseteq A$ is \emph{externally stable} in $T$ with respect to tournament solution $S$ (or \emph{$S$-stable}) if $a\notin S(B\cup\{a\})$ for all~$a\in A\setminus B$.  The set of $S$-stable sets for a given tournament $T=(A,{\succ})$ will be denoted by $\mathcal{S}_S(T)=\{ B\subseteq A\mid B \text{ is $S$-stable in $T$} \}$.
\end{definition}
Externally stable sets are guaranteed to exist since the set of all alternatives $A$ is trivially $S$-stable in $(A,{\succ})$ for every $S$.  We say that a set $B\subseteq A$ is \emph{internally stable} with respect to $S$ if $S(B)=B$.
We will focus on external stability for now because we will see later that certain conditions imply the existence of a unique minimal externally stable set, which also satisfies internal stability. We define $\ms{S}(T)$ to be the tournament solution that returns the union of all \emph{inclusion-minimal} $S$-stable sets in $T$, \ie the union of all $S$-stable sets that do not contain an $S$-stable set as a proper subset.

\begin{definition}\label{def:minimalsstable}
Let $S$ be a tournament solution. Then, the tournament solution $\ms{S}$ is defined as
\[
\ms{S}(T) = \bigcup \min_\subseteq(\mathcal{S}_S(T))\text{.}
\]
\end{definition}
It is easily verified that $\ms{S}$ is well-defined as a tournament solution as there are no $S$-stable sets that do not contain the Condorcet winner whenever one exists.
We will only be concerned with tournament solutions~$S$ that (presumably) admit a \emph{unique} minimal $S$-stable set in any tournament. It turns out it is precisely this property that is most difficult to prove for all but the simplest tournament solutions.
A tournament $T$ contains a \emph{unique} minimal $S$-stable set if and only if $\mathcal{S}_S(T)$ is a \emph{directed} set with respect to set inclusion, \ie for all sets $B,C\in\mathcal{S}_S(T)$ there is a set $D\in\mathcal{S}_S(T)$ contained in both $B$ and $C$. We say that $\mathcal{S}_S$ is directed when $\mathcal{S}_S(T)$ is a directed set for all tournaments $T$. 
Throughout this \doc, directedness of a set of sets $\mathcal{S}$ is shown by proving the stronger property of closure under intersection, \ie $B\cap C\in \mathcal{S}$ for all $B,C\in \mathcal{S}$.
A set of sets $\mathcal{S}$ \emph{pairwise intersects} if $B\cap C\ne \emptyset$ for all $B,C\in\mathcal{S}$.
We will prove that, for every concept of qualified subsets $\mathcal{Q}$, $\mathcal{S}_{S_\mathcal{Q}}$ is closed under intersection if and only if $\mathcal{S}_{S_\mathcal{Q}}$ pairwise intersects. 
In order to improve readability, we will use the short notation $\mathcal{S}_{\mathcal{Q}}$ for $\mathcal{S}_{S_\mathcal{Q}}$.

\begin{lemma}\label{lem:pairwise}
 Let $\mathcal{Q}$ be a concept of qualified subsets. Then,  $\mathcal{S}_{\mathcal{Q}}$ is closed under intersection if and only if $\mathcal{S}_{\mathcal{Q}}$ pairwise intersects.
\end{lemma}
\begin{proof}
The direction from left to right is straightforward since the empty set is not stable.
The opposite direction is shown by contraposition, \ie we prove that $\mathcal{S}_{\mathcal{Q}}$ does not pairwise intersect if $\mathcal{S}_{\mathcal{Q}}$ is not closed under intersection.
Let $T=(A,\succ)$ be a tournament and $B_1,B_2 \in \mathcal{S}_{\mathcal{Q}}(T)$ be two sets such that $C=B_1\cap B_2\not\in \mathcal{S}_{\mathcal{Q}}(T)$. 
Since $C$ is not $S_{\mathcal{Q}}$-stable, there has to be $a\in A\setminus C$ such that $a\in S_{\mathcal{Q}}(C\cup\{a\})$. In other words, there has to be a set $Q\subseteq C$ such that $[Q,a]\in \max_\subseteq(\mathcal{Q}(C\cup\{a\}))$.
Define
\[
B_1'=\{b\in B_1\mid b\succ Q\} \text{ and } B_2'=\{b\in B_2\mid b\succ Q\}\text{.}
\]
Clearly, $(B_1'\setminus B_2')\cap C=\emptyset$ and $(B_2'\setminus B_1')\cap C=\emptyset$.
Assume without loss of generality that $a\not\in B_1$. It follows from the stability of $B_1$, that $B_1$ has to contain an alternative $b_1$ such that  $b_1\succ [Q,a]$. Hence, $B_1'$ is not empty.
Next, we show that $B_1'\cap B_2'=\emptyset$. Assume for contradiction that there is some $b\in B_1'\cap B_2'$. If $b\succ a$, independence implies that $[Q\cup\{a\},b]\in \mathcal{Q}(C\cup\{a\})$, which contradicts the fact that $[Q,a]$ is a maximal qualified subset in $C\cup\{a\}$. If, on the other hand, $a\succ b$, the set $[Q\cup\{b\},a]$ is isomorphic to $[Q\cup\{a\},b_1]$, which is a qualified subset of $B_1\cup\{a\}$. Thus, $[Q\cup\{b\},a]\in \mathcal{Q}(C\cup\{a\})$, again contradicting the maximality of $[Q,a]$. 
Independence, the isomorphism of $[Q,a]$ and $[Q,b_1]$, and the stability of $B_2$ further require that there has to be an alternative $b_2\in B_2$ such that $b_2\succ [Q,b_1]$. Hence, $B_1'$ and $B_2'$ are disjoint and non-empty.

Let $a'\in B_2'$ and $R$ be a maximal subset of $B_1'\cup Q$ such that $[R,a']\in \mathcal{Q}(B_1'\cup Q \cup \{a'\})$. 
We claim that $Q$ has to be contained in $R$. Assume for contradiction that there exists some $b\in Q\setminus R$. Clearly, $[Q,a]$ and $[Q,a']$ are isomorphic. It therefore follows from independence that $[Q,a']\in \mathcal{Q}(B_1'\cup Q \cup \{a'\})$ and from closure that $[(Q\cap R)\cup\{b\},a']\in \mathcal{Q}(B_1'\cup Q \cup \{a'\})$. 
Due to the stability of $B_1$, $[R,a']$ is not a maximal qualified subset in $B_1\cup\{a'\}$, \ie there exists a qualified subset that contains more elements. We may thus merge the qualified subsets $[R,a']$ and $[(Q\cap R)\cup\{b\},a']$ according to the fusion condition because $R\setminus Q\succ Q$ and consequently $R\setminus Q\succ (Q\cap R)\cup\{b\}$. We then have that $[R\cup\{b\},a']\in \mathcal{Q}(B_1'\cup Q \cup \{a'\})$, which yields a contradiction because $R$ was assumed to be a maximal set such that $[R,a']\in \mathcal{Q}(B_1'\cup Q \cup \{a'\})$. Hence, $Q\subseteq R$. 
Due to the stability of $B_1$ in $T$, there has to be a $c\in B_1$ such that $c\succ [R,a']$. Since $B_1'$ contains all alternatives  in $B_1$ that dominate $Q\subseteq R$, it also contains $c$. Independence then implies that $[R,a']\not\in \max_\subseteq(\mathcal{Q}_k(B_1'\cup Q \cup \{a'\}))$.

Thus, $B_1'\cup Q$ is stable in $B_1'\cup B_2'\cup Q$. Since $Q$ is contained in every maximal set $R\subseteq B_1'\cup Q$ such that $[R,a']\in \mathcal{Q}(B_1'\cup Q \cup \{a'\})$ for some $a'\in B_2'$, $B_1'$ (and by an analogous argument $B_2'$) remains stable when removing $Q$. This completes the proof because $B_1'$ and $B_2'$ are two disjoint $S_\mathcal{Q}$-stable sets in $B_1'\cup B_2'$.
\end{proof}

\citeauthor{Dutt88a} has shown by induction on the tournament order that tournaments admit no disjoint $S_\mathcal{M}$-stable sets (so-called \emph{covering sets}).
\begin{theorem}[\citealp{Dutt88a}]\label{thm:Dutta}
$\mathcal{S}_{\mathcal{M}}$ pairwise intersects.
\end{theorem}
\citet{Dutt88a} also showed that covering sets are closed under intersection, which now also follows from \lemref{lem:pairwise}.\footnote{As \citeauthor{Dutt88a}'s definition requires a stable set to be internally and externally stable, he actually proves that the intersection of any pair of coverings sets \emph{contains} a covering set. A simpler proof, which shows that \emph{externally} $S_\mathcal{M}$-stable sets are closed under intersection, is given by \citet{Lasl97a}.}

Naturally, finer solution concepts also yield smaller minimal stable sets (if their uniqueness is guaranteed).

\begin{proposition}\label{pro:stableinclusion}
Let $S$ and $S'$ be two tournament solutions such that $\mathcal{S}_{S'}$ is directed and $S'\subseteq S$. Then, $\ms{S}'\subseteq \ms{S}$ and $\mathcal{S}_S$ pairwise intersects.  
\end{proposition}
\begin{proof}
The statements follow from the simple fact that every $S$-stable set is also $S'$-stable.  Let $B\subseteq A$ be a minimal $S$-stable set in tournament $(A,{\succ})$. Then, $a\not\in S(B\cup\{a\})$ for every $a\in A\setminus B$ and, due to the inclusion relationship, $a\not\in S'(B\cup\{a\})\subseteq S(B\cup\{a\})$. As a consequence, $B$ is $S'$-stable and has to contain the unique minimal $S'$-stable set since $\mathcal{S}_{S'}$ is directed.
$\mathcal{S}_S$ pairwise intersects because two disjoint $S$-stable sets would also be $S'$-stable, which contradicts the directedness of $\mathcal{S}_{S'}$.
\end{proof}

As a corollary of the previous statements, the set of $S_{\mathcal{M}_k}$-stable sets for every $k$ is closed under intersection.

\begin{theorem}\label{thm:uniquekcovering}
$\mathcal{S}_{\mathcal{M}_k}$ is closed under intersection for all $k\in\mathbb{N}$.
\end{theorem}
\begin{proof}
Let $k\in\mathbb{N}$. We know from \propref{pro:inclusions} that $S_\mathcal{M}\subseteq S_{\mathcal{M}_k}$ and from \thmref{thm:Dutta} and \lemref{lem:pairwise} that $\mathcal{S}_\mathcal{M}$ is directed. \propref{pro:stableinclusion} implies that $\mathcal{S}_{\mathcal{M}_k}$ pairwise intersects. The statement then straightforwardly follows from \lemref{lem:pairwise}.
\end{proof}

Interestingly, $\mathcal{S}_{\mathcal{M}_2}$, the set of all dominant sets, is not only closed under intersection, but in fact totally ordered with respect to set inclusion. 

We conjecture that the set of all $S_{\mathcal{M}^*}$-stable sets also pairwise intersects and thus admits a unique minimal element. However, the combinatorial structure of transitive subtournaments within tournaments is extraordinarily rich \citep[see, \eg][]{Woeg03a,GaMn10a} and it seems that a proof of the conjecture would be significantly more difficult than \citeauthor{Dutt88a}'s. 
As will be shown in \charef{sec:teq}, the conjecture that $S_{\mathcal{M}^*}$-stable sets pairwise intersect is a weakened version of a conjecture by \citet{Schw90a}.

\begin{conjecture}\label{con:me}
$\mathcal{S}_{\mathcal{M}^*}$ is closed under intersection.
\end{conjecture}
Using \lemref{lem:pairwise}, the conjecture entails that $\mathcal{S}_{\mathcal{M}_k^*}$ for all $k\in\mathbb{N}$ is also closed under intersection.
Since tournaments with less than four alternatives may not contain a maximal element and a cycle at the same time, this trivially holds for $k\le 3$.
The weakest version of \conref{con:me} that is no implied by \thmref{thm:Dutta} is that $\mathcal{S}_{\mathcal{M}_4^*}$ is closed under intersection. We were able to show this by reducing it to a large, but finite, number of cases that were checked using a computer. Unfortunately, this exercise did not yield any insight on how to prove \conref{con:me}.

Two well-known examples of minimal stable sets are the top cycle of a tournament, which is the minimal stable set with respect to $S_{\mathcal{M}_2}$, and the minimal covering set, which is the minimal stable set with respect to $S_{\mathcal{M}}$.

\paragraph{Minimal dominant set.} The \emph{minimal dominant set} (or top cycle) of a tournament $T=(A,{\succ})$ is given by $\tc(T)=\ms{S}_{\mathcal{M}_2}(T)=\ms{\cnl}$, \ie it is the smallest set $B$ such that $B\succ A\setminus B$ \citep{Good71a,Smit73a}.

\paragraph{Minimal covering set.} The \emph{minimal covering set} of a tournament $T$ is given by $\mc(T)=\ms{S}_{\mathcal{M}}(T)=\ms{\uc}$, \ie it is the smallest set $B$ such that for all $a\in A\setminus B$, there exists $b\in B$ so that every alternative in $B$ that is dominated by $a$ is also dominated by~$b$ \citep{Dutt88a}.
\bigskip

The proposed methodology also suggests the definition of a new tournament solution that has not been considered before in the literature.

\paragraph{Minimal extending set.} The \emph{minimal extending set} of a tournament $T$ is given by $\me[T]=\ms{S}_{\mathcal{M}^*}(T)=\ms{\ba}$, \ie it is the smallest set $B$ such that no $a\in A\setminus B$ is the maximal element of a maximal transitive subset in $B\cup\{a\}$. 
\bigskip

The minimal extending set will be further analyzed in Section~\ref{sec:me}.

\subsection{Properties of Minimal Stable Sets}

If $\mathcal{S}_S$ is directed---and we will only be concerned with tournament solutions $S$ for which this is (presumably) the case---$\ms{S}$ satisfies a number of desirable properties.

\begin{proposition}\label{pro:stableaiz}
Let $S$ be a tournament solution such that $\mathcal{S}_S$ is directed. Then, $\ms{S}$ satisfies \wsp and \iua.
\end{proposition}
\begin{proof}
Clearly, any minimal $S$-stable set $B$ remains $S$-stable when losing alternatives are removed or when edges between losing alternatives are modified. In the latter case, $B$ also remains minimal. In the former case, the minimal $S$-stable set is contained in~$B$.
\end{proof}

It can be shown that sets that are stable within a stable set are also stable in the original tournament when the underlying tournament solution is defined via a concept of qualified subsets $\mathcal{Q}$. This lemma will prove very useful when analyzing $\ms{S}_\mathcal{Q}$.

\begin{lemma}\label{lem:stablestable}
Let $T=(A,\succ)$ be a tournament and $\mathcal{Q}$ a concept of qualified subsets. Then, $\mathcal{S}_\mathcal{Q}(B)\subseteq \mathcal{S}_\mathcal{Q}(A)$ for all $B\in \mathcal{S}_\mathcal{Q}(A)$.
\end{lemma}
\begin{proof}
We prove the statement by showing that the following implication holds for all $B\subseteq A$, $C\in \mathcal{S}_\mathcal{Q}(B)$, and $a\in A$:
\[
\text{if }a\not\in S_\mathcal{Q}(B\cup\{a\}) \text{ then } a\not\in S_\mathcal{Q}(C\cup\{a\})\text{.}
\]
To see this, let $a\not\in S_\mathcal{Q}(B\cup\{a\})$ and assume for contradiction that there exist $[Q,a]\in \max_\subseteq\mathcal{Q}(C\cup\{a\})$. Then there has to be $b\in B$ such that $[Q\cup\{a\},b]\in\mathcal{Q}(B\cup\{a\})$ because $[Q,a]\not\in \max_\subseteq \mathcal{Q}(B\cup\{a\})$. Now, if $b\in C$, closure and independence imply
that $[Q\cup\{a\},b]\in \mathcal{Q}(C\cup\{a\})$, contradicting the maximality of $[Q,a]$. If, on the other hand, $b\in B\setminus C$, then there has to be $c\in C$ such that $[Q\cup\{b\},c]\in \mathcal{Q}(C\cup\{b\})$. No matter whether $c\succ a$ or $a\succ c$, $Q\cup\{a,c\}$ is isomorphic to $[Q\cup\{b\},c]$ and thus also a qualified subset, which again contradicts the assumption that $[Q,a]$ was maximal.
\end{proof}

We are now ready to show a number of appealing properties of \emph{unique} minimal stable sets when the underlying solution concept is defined via qualified subsets.
\todo{rewrite using $\widehat{\gamma}$, self-stability, \etc}
\begin{theorem}\label{thm:minstableproperties}
Let $\mathcal{Q}$ be a concept of qualified subsets such that $\mathcal{S_Q}$ is directed.  Then,
\begin{enumerate}[label=\textit{(\roman*)}]
 \item $\ms{S}_\mathcal{Q}\subseteq S_\mathcal{Q}^\infty$,
 \item\label{itm:internalstability} $S_\mathcal{Q}(\ms{S}_\mathcal{Q}(T)\cup\{a\})=\ms{S}_\mathcal{Q}(T)$ for all tournaments $T=(A,{\succ})$ and $a\in A$ (in particular, $\ms{S}_\mathcal{Q}(T)$ is internally stable),
 \item $\ms{S}_\mathcal{Q}$ satisfies \ssp, and
 \item $\ms{\ms{S}}_\mathcal{Q}=\ms{S}_\mathcal{Q}$.
\end{enumerate}
\end{theorem}
\begin{proof}
Let $T=(A,\succ)$ be a tournament. 
The first statement of the theorem is shown by proving by induction on $k$ that $S_\mathcal{Q}^k(T)$ is an $S_\mathcal{Q}$-stable set.
For the basis, let $B=S_\mathcal{Q}(T)$.  Then, $S_\mathcal{Q}(B\cup \{a\})\subseteq B$ for every $a\in A\setminus B$ due to \wsp of $S_\mathcal{Q}$ (\propref{pro:qwspmon}) and thus $B$ is $S_\mathcal{Q}$-stable. 
Now, assume that $B=S_\mathcal{Q}^{k}(T)$ is $S_\mathcal{Q}$-stable and let $C=S_\mathcal{Q}(B)$. Again, \wsp implies that $a\not\in S_\mathcal{Q}(C\cup \{a\})$ for every $a\in B\setminus C$, \ie $C\in\mathcal{S}_\mathcal{Q}(B)$. We can thus directly apply \lemref{lem:stablestable} to obtain that $C=S_\mathcal{Q}^{k+1}(T)\in\mathcal{S}_\mathcal{Q}(T)$. As the minimal $S_\mathcal{Q}$-stable set is contained in every $S_\mathcal{Q}$-stable set, the statement follows.

Regarding internal stability, assume for contradiction that $S_\mathcal{Q}(\ms{S}_\mathcal{Q}(T))\subset \ms{S}_\mathcal{Q}(T)$. However, \lemref{lem:stablestable} implies that $S_\mathcal{Q}(\ms{S}_\mathcal{Q}(T))$ is $S_\mathcal{Q}$-stable, contradicting the minimality of $\ms{S}_\mathcal{Q}(T)$.  The remainder of the second statement follows straightforwardly from internal stability. If $S_\mathcal{Q}(\ms{S}_\mathcal{Q}(T)\cup \{a\})=C\subset \ms{S}_\mathcal{Q}(T)$ for some $a\in A\setminus \ms{S}_\mathcal{Q}(T)$, \wsp implies that $S_\mathcal{Q}(\ms{S}_\mathcal{Q}(T))\subseteq C$, contradicting internal stability.

Regarding \ssp, let $B=\ms{S}_\mathcal{Q}(T)$ and assume for contradiction that $C=\ms{S}_\mathcal{Q}(A')\subset B$ for some $A'$ with $B\subseteq A'\subset A$. Clearly, $C$ is $S_\mathcal{Q}$-stable not only in $A'$ but also in $B$, which implies that $C\in \mathcal{S}_{\mathcal{Q}}(B)$. According to \lemref{lem:stablestable}, $C$ is also contained in $\mathcal{S}_\mathcal{Q}(A)$, contradicting the minimality of $\ms{S}_\mathcal{Q}(T)$.

Finally, for $\ms{\widehat{S}}_\mathcal{Q}(T)=\ms{S}_\mathcal{Q}(T)$, we show that every $S_\mathcal{Q}$-stable set is $\ms{S}_\mathcal{Q}$-stable and that every \emph{minimal} $\ms{S}_\mathcal{Q}$-stable set is $S_\mathcal{Q}$-stable set.  The former follows from $\ms{S}_\mathcal{Q}(T)\subseteq S_\mathcal{Q}(T)$, which is a consequence of the first statement of this theorem.  For the latter statement, let $B\in \min_\subseteq(\mathcal{S}_{\ms{S}_\mathcal{Q}}(T))$.  
We first show that $\ms{S}_\mathcal{Q}(B\cup \{a\})=B$ for all $a\in A\setminus B$.
Assume for contradiction that $\ms{S}_\mathcal{Q}(B\cup \{a\})=C\subset B$ for some $a\in A\setminus B$. Since $\ms{S}_\mathcal{Q}$ satisfies \ssp, $\ms{S}_\mathcal{Q}(B\cup \{a\})=C$ for all $a\in A\setminus B$. As a consequence, $C$ is $\ms{S}_\mathcal{Q}$-stable in $B\cup\{a\}$ for all $a\in A\setminus B$ and, due to the definition of stability, also in $A$. This contradicts the assumption that $B$ was the \emph{minimal} $\ms{S}$-stable set. Hence, $\ms{S}_\mathcal{Q}(B\cup \{a\})=B$ for all $a\in A\setminus B$. By definition of $\ms{S}_\mathcal{Q}$, this implies that $a\not\in S_\mathcal{Q}(B\cup\{a\})$ and thus that $B$ is $S_\mathcal{Q}$-stable.
\end{proof}
The second statement of Theorem~\ref{thm:minstableproperties} allows us to characterize stable sets using the fixed-point formulation mentioned at the beginning of this section, which unifies internal and external stability.

\begin{corollary}\label{cor:fpdef}
Let $\mathcal{Q}$ be a concept of qualified subsets such that $\mathcal{S}_\mathcal{Q}$ is directed and $T=(A,{\succ})$ a tournament. Then, 
\[\ms{S}_\mathcal{Q}(T)=\min_\subseteq \{B\subseteq A \mid B=\bigcup_{a\in A} S_\mathcal{Q}(B\cup\{a\})\}\text{.}\] 
\end{corollary}

There may very well be more than one internally and externally $S_\mathcal{Q}$-stable set in a tournament.
For example, the proof of Theorem~\ref{thm:minstableproperties} implies that $S_\mathcal{Q}^\infty(T)$ is internally and externally $S_\mathcal{Q}$-stable.

We have already seen that $\ms{S}_\mathcal{Q}$ satisfies some of the basic properties defined in Section~\ref{sec:properties}. It further turns out that $\ms{S}$ inherits monotonicity and composition-consistency from $S$.

\begin{proposition}\label{pro:stablemon}
Let $S$ be a tournament solution such that $\mathcal{S}_S$ is directed and $S$ satisfies \mon. Then, $\ms{S}$ satisfies \mon as well.
\end{proposition}
\begin{proof}
Let $T=(A,{\succ})$ be a tournament with $a,b\in A$, $a\in \ms{S}(T)$, and $b\succ a$, and let the relation $\succ'$ be identical to $\succ$ except that $a\succ' b$. Denote $T'=(A,\succ')$ and assume for contradiction that $a\not\in \ms{S}(T')$. Then, there has to be a minimal $S$-stable set $B\subseteq A\setminus\{a\}$ in $T'$. We show that $B$ is also $S$-stable in $T$, a contradiction. If $b\not\in B$, this would clearly be the case because $\ms{S}$ satisfies $\iua$. If, on the other hand, $b\in B$, the only reason for $B$ not to be $S$-stable in $T$ is that $a\in S((B\cup\{a\}),\succ)$. However, monotonicity of $S$ then implies that $a\in S((B\cup\{a\}),\succ')$, which is a contradiction because $B$ is $S$-stable in $T'$.
\end{proof}

\begin{proposition}\label{pro:stablecomp}
Let $S$ be a tournament solution that satisfies \com. Then, $\ms{S}$ satisfies \com as well.
\end{proposition}
\begin{proof}
Let $S$ be a composition-consistent tournament solution and $T=(A,{\succ})=\Pi(\tilde{T},T_1,\dots,T_k)$ a product tournament with $T=(\{1,\dots,k\},\tilde{\succ})$, $T_1=(B_1,\succ_1)$, \dots, $T_k=(B_k,\succ_k)$. For a subset $C$ of $A$, let $C_i=C\cap B_i$ for all $i\in\{1,\dots,k\}$ and $\tilde{C}=\bigcup_{i\colon C_i\ne \emptyset} \{i\}$. We will prove that $C\subseteq A$ is $S$-stable if and only if 
\begin{enumerate}[label=\textit{(\roman*)}]
\item $\tilde{C}$ is $S$-stable in $\tilde{T}$, and 
\item $C_i$ is $S$-stable in $T_i$ for all $i\in\{1,\dots,k\}$. 
\end{enumerate}
Consider an arbitrary alternative $a\in A\setminus C$. For $C$ to be $S$-stable, $a$ should not be contained in $S(C\cup\{a\})$. Since $S$ is composition-consistent, $a$ may be excluded for two reasons. First, $a$ may be contained in an unchosen component, \ie $a\in B_i$ such that $i\not\in S(\tilde{C}\cup \{i\})$. Secondly, $a$ may not be selected despite being in a chosen component, \ie $a\in B_i$ such that $i\in S(\tilde{C}\cup \{i\})$ and $a\not\in S(C_i\cup\{a\})$. This directly establishes the claim above and consequently that $\ms{S}$ is composition-consistent.
\end{proof}

The previous propositions and theorems allow us to deduce several known statements about $\tc$ and $\mc$, in particular that both concepts satisfy all basic properties and that $\mc$ is a refinement of $\uc^\infty$ and satisfies \com.

We conclude this section by generalizing the axiomatization of the minimal covering set \citep{Dutt88a} to abstract minimal stable sets. One of the cornerstones of the axiomatization is $S$-exclusivity, which prescribes under which circumstances a single element may be dismissed from the choice set.\footnote{$\uc$-exclusivity is the property $\gamma^{**}$ used in the axiomatization of $\mc$ \citep{Lasl97a}.}

\begin{definition}
 A tournament solution $S'$ satisfies \emph{$S$-exclusivity} if, for every tournament $T=(A,{\succ})$,  $S'(T)=A\setminus \{a\}$ implies that $a\not\in S(A)$.
\end{definition}

If $S$ always admits a \emph{unique} minimal $S$-stable set and $\ms{S}$ satisfies \ssp, which is always the case if $S$ is defined via qualified subsets, then $\ms{S}$ can be characterized by \ssp, $S$-exclusivity, and inclusion-minimality.

\begin{proposition}\label{pro:stablecharac}
Let $S$ be a tournament solution such that $\mathcal{S}_S$ is directed and $\ms{S}$ satisfies \ssp. Then, $\ms{S}$ is the finest tournament solution satisfying \ssp and $S$-exclusivity.
\end{proposition}
\begin{proof}
Let $S$ be a tournament solution as desired and $S'$ a tournament solution that satisfies \ssp and $S$-exclusivity.  We first prove that $\ms{S}\subseteq S'$ by showing that $S'(T)$ is $S$-stable for every tournament $T=(A,{\succ})$. Let $B=S'(T)$ and $a\in A\setminus B$. It follows from \ssp that $S'(B\cup \{a\})=B$ and from $S$-exclusivity that $a\not\in S(B\cup\{a\})$, which implies that $B$ is $S$-stable. Since $\ms{S}(T)$ is the unique inclusion-minimal $S$-stable set, it has to be contained in all $S$-stable sets.
The statement now follows from the fact that $\ms{S}$ satisfies \ssp and $S$-exclusivity.
\end{proof}

Hence, $\tc$ is the finest tournament solution satisfying \ssp and $\cnl$-exclusivity, $\mc$ is the finest tournament solution satisfying \ssp and $\uc$-exclusivity, and $\me$ is the finest tournament solution satisfying \ssp and $\ba$-exclusivity if \conref{con:me} holds.

\subsection{The Minimal Extending Set}
\label{sec:me} 

As mentioned in Section~\ref{sec:directedness}, the minimal extending set is a new tournament solution that has not been considered before. In analogy to $\uc$-stable sets, which are known as covering sets, we will call $\ba$-stable sets \emph{extending sets}. $B$ is an extending set of tournament $T=(A,{\succ})$ if, for all $a\not\in B$, every transitive path (or so-called Banks trajectory) in $B\cup\{a\}$ with maximal element $a$ can be extended, \ie there is $b\in B$ such that $b$ dominates every element on the path. In other words, $B\subseteq A$ is an \emph{extending set} if for all~$a\in A\setminus B$, $a\notin \ba[B\cup\{a\}]$. 

If Conjecture~\ref{con:me} is correct, $\me$ satisfies all properties defined in Section~\ref{sec:properties} and is a refinement of $\ba$ due to Propositions~\ref{pro:stableaiz} and~\ref{pro:stablemon} and Theorem~\ref{thm:minstableproperties}.
Assuming that Conjecture~\ref{con:me} holds, Proposition~\ref{pro:stableinclusion} furthermore implies that $\me$ is a refinement of $\mc$ since every covering set is also an extending set.  We refer to \figref{fig:mcme} for an example tournament $T$ where $\me[T]$ happens to be \emph{strictly} contained in $\mc[T]$.\footnote{This is also the case for a tournament of order eight given by \citet{Dutt90a}.}

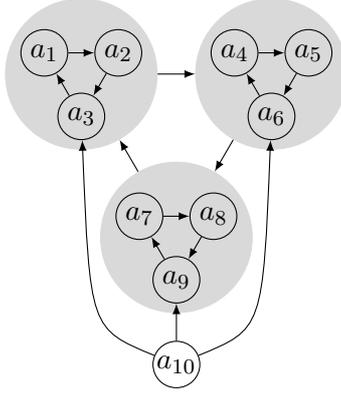
\begin{figure}[htb]
  \centering\figuresize
  \def\nd{0.66em} 
  \def\ra{0.8em} 
  \begin{tikzpicture}[node distance=\nd,vertex/.style={circle,draw,minimum size=2*\ra,inner sep=0pt}]
	\draw
    node[circle,fill=black!15,draw=black!15,minimum size=3*\nd+4*\ra] (c1) {} 
    +(150:\nd+\ra) node[vertex] (a1) {$a_1$} 
    +(30:\nd+\ra) node[vertex] (a2)  {$a_2$}
    +(270:\nd+\ra) node[vertex] (a3)  {$a_3$} ++(0:5*\nd+4*\ra)
    node[circle,fill=black!15,draw=black!15,minimum size=3*\nd+4*\ra] (c2) {} 
    +(150:\nd+\ra) node[vertex] (a4) {$a_4$} 
    +(30:\nd+\ra) node[vertex] (a5)  {$a_5$}
    +(270:\nd+\ra) node[vertex] (a6)  {$a_6$} ++(240:5*\nd+4*\ra)
    node[circle,fill=black!15,draw=black!15,minimum size=3*\nd+4*\ra] (c3) {} 
    +(150:\nd+\ra) node[vertex] (a7) {$a_7$} 
    +(30:\nd+\ra) node[vertex] (a8)  {$a_8$}
    +(270:\nd+\ra) node[vertex] (a9)  {$a_9$} ++(270:3*\nd+3*\ra) 
    node[vertex] (a10) {$a_{10}$};

	\draw [-latex] (a1) to (a2);
	\draw [-latex] (a2) to (a3);
	\draw [-latex] (a3) to (a1);
	\draw [-latex] (a4) to (a5);
	\draw [-latex] (a5) to (a6);
	\draw [-latex] (a6) to (a4);
	\draw [-latex] (a7) to (a8);
	\draw [-latex] (a8) to (a9);
	\draw [-latex] (a9) to (a7);
	\draw [-latex] (c1) to (c2);
	\draw [-latex] (c2) to (c3);
	\draw [-latex] (c3) to (c1);
	\draw [-latex] (a10) .. controls +(157.5:4*\ra) .. (a3);
	\draw [-latex] (a10) to (a9);
	\draw [-latex] (a10) .. controls +(22.5:4*\ra) .. (a6);
 \end{tikzpicture}
 \caption{Example tournament $T=(A,{\succ})$ where $\mc$ and $\me$ differ ($\mc(T)=A$ and $\me(T)=A\setminus\{a_{10}\}$). $a_{10}$ only dominates $a_3$, $a_6$, and $a_9$.}
 \label{fig:mcme}
\end{figure}

A remarkable property of $\me$ is that, just like $\ba$, it is capable of ruling out alternatives in \emph{regular} tournaments, \ie it satisfies \irr. No irregular tournament solution is known to satisfy all four basic properties. However, if Conjecture~\ref{con:me} were true, $\me$ would be such a concept.


\begin{savequote}
\sffamily
I represent collective-choice processes by mathematical objects that depict relatively little of those processes' structural complexity; this enhances the generality of my conclusions.
\qauthor{T. Schwartz, 1986}
 \end{savequote}


\section{Retentiveness and Stability}\label{sec:teq}

Motivated by cooperative majority voting, \citet{Schw90a} introduced a tournament solution based on a notion he calls \emph{retentiveness}.
It turns out that retentiveness bears some similarities to stability. For example, the top cycle can be represented as a minimal stable set as well as a minimal retentive set, albeit using different underlying tournament solutions.

\subsection{The Tournament Equilibrium Set}\label{sec:teqsection}

The intuition underlying retentive sets is that alternative $a$ is only ``properly'' dominated by alternative $b$ if $b$ is chosen among $a$'s dominators by some underlying tournament solution~$S$. A set of alternatives is then called $S$-retentive if none of its elements is properly dominated by some outside alternative with respect to $S$.\footnote{In analogy to the discussion at the beginning of \secref{sec:directedness}, a set of alternatives $B$ may be called retentive if it consists precisely of those alternatives not upset by $B$. Here, $a$ is upset by $B$ if $a\not\in \bigcup_{b\in B} S(\dom(b))$ for some underlying solution concept $S$.}

\begin{definition}\label{def:sretentive}
Let $S$ be a tournament solution and $T=(A,{\succ})$ a tournament.  Then, $B\subseteq A$ is \emph{retentive} in $T$ with respect to tournament solution $S$ (or \emph{$S$-retentive}) if $B\ne \emptyset$ and $S(\dom(b))\subseteq B$ for all~$b\in B$ such that $\dom(b)\ne \emptyset$. The set of $S$-retentive sets for a given tournament $T=(A,{\succ})$ will be denoted by $\mathcal{R}_S(T)=\{ B\subseteq A\mid B \text{ is $S$-retentive in $T$} \}$.
\end{definition}
$S$-retentive sets are guaranteed to exist since the set of all alternatives $A$ is trivially $S$-retentive in $(A,{\succ})$ for all tournament solutions $S$. In analogy to Definition~\ref{def:minimalsstable}, the union of minimal $S$-retentive sets defines a tournament solution.

\begin{definition}\label{def:minimalsretentive}
Let $S$ be a tournament solution. Then, the tournament solution $\mr{S}$ is defined as
\[
\mr{S}(T) = \bigcup \min_\subseteq(\mathcal{R}_S(T))\text{.}
\]
\end{definition}
It is easily verified that $\mr{S}$ is well-defined as a tournament solution as there are no $S$-retentive sets that do not contain the Condorcet winner whenever one exists.

As an example, consider the tournament solution $S_{\mathcal{M}_1}$ that always returns all alternatives, \ie $S_{\mathcal{M}_1}((A,{\succ}))=A$.
The unique minimal retentive set with respect to $S_{\mathcal{M}_1}$ is the top cycle, that is $\tc=\ms{S}_{\mathcal{M}_2}=\mr{S}_{\mathcal{M}_1}$.

\citeauthor{Schw90a} introduced retentiveness in order to recursively define the \emph{tournament equilibrium set} ($\teq$) as the union of minimal $\teq$-retentive sets. This recursion is well-defined because the order of the dominator set of any alternative is strictly smaller than the order of the original tournament.

\begin{definition}[\citealp{Schw90a}]\label{def:teq}
The \emph{tournament equilibrium set} (TEQ) of a tournament $T$ is defined recursively as $\teq(T)=\mr{\teq}(T)$.
\end{definition}
In other words, $\teq$ is the unique fixed point of the $\circ$-operator.

\citeauthor{Schw90a} conjectured that every tournament admits a \emph{unique} minimal $\teq$-retentive set. Despite several attempts to prove or disprove this statement \citep[\eg][]{LLL93a,Houy09a}, the statement has remained a conjecture. A recent computer analysis failed to find a counter-example in all tournaments of order $12$ or less and a fairly large number of random tournaments \citep{BFHM09a}.

\begin{conjecture}[\citealp{Schw90a}]\label{con:teq}
$\mathcal{R}_{\teq}$ is directed.
\end{conjecture}

It is easily appreciated that the non-empty intersection of two $S$-retentive sets is also $S$-retentive. 
%
%
%
As a consequence, Conjecture~\ref{con:teq} is equivalent to the statement that there are no two disjoint $\teq$-retentive sets in any tournament.
Unfortunately, and somewhat surprisingly, it is not known whether $\teq$ satisfies \emph{any} of the basic properties defined in \secref{sec:properties}. However, \citet{LLL93a} and \citet{Houy09a} 
 have shown that $\teq$ satisfies any of the basic properties if and only if $\mathcal{R}_{\teq}$ is directed, and is strictly contained in $\mc$ if $\mathcal{R}_{\teq}$ is directed. We will strengthen the last statement by showing that $\teq$ is also strictly contained in $\me$ if $\mathcal{R}_{\teq}$ is directed.

\subsection{Inclusion of TEQ in ME}

In \secref{sec:maxmaxqualified}, the Banks set was characterized as the finest tournament solution satisfying strong retentiveness. It turns out that, if $\mathcal{R}_{\teq}$ is directed, $\teq$ is the finest tournament solution satisfying a very natural weakening of strong retentiveness, where the inclusion of the choice sets of dominator sets is only required to hold for alternatives contained in the original choice set.

\begin{definition}
A tournament solution~$S$ satisfies (weak) \emph{retentiveness} if $S(\dom(a))\subseteq S(T)$ for all tournaments $T$ and $a\in S(T)$.
\end{definition}
In other words, a tournament solution $S$ satisfies retentiveness if and only if $S(T)$ is $S$-retentive for all tournaments $T$.
It follows from Schwartz's axiomatization of $\teq$ that $\teq$ is the finest tournament solution satisfying retentiveness, given that Conjecture~\ref{con:teq} is true.

\begin{theorem}[\citealp{Schw90a}]\label{thm:teqdcon}
$\teq$ is the finest tournament solution satisfying retentiveness if $\mathcal{R}_{\teq}$ is directed.
\end{theorem}

The fact that $S_\mathcal{Q}$ satisfies strong retentiveness for every concept of qualified subsets $\mathcal{Q}$ can be used to show that every $S_\mathcal{Q}$-stable set is $\ms{S}_\mathcal{Q}$-retentive, which has a number of useful consequences.

\begin{lemma}\label{lem:selfret}
Let $\mathcal{Q}$ be a concept of qualified subsets such that $\mathcal{S}_{S_\mathcal{Q}}$ is directed for all tournaments of order $n$ or less.
Then, $\mathcal{S}_{S_\mathcal{Q}}(T)\subseteq \mathcal{R}_{\ms{S}_\mathcal{Q}}(T)$ for all tournaments $T$ of order $n+1$.
\end{lemma}
\begin{proof}
Let $T=(A,\succ)$ be a tournament of order $n+1$, $B\in \mathcal{S}_{S_\mathcal{Q}}(T)$, $b\in B$, and $C=\dom_B(b)$. We show that $C$ is $S_\mathcal{Q}$-stable in $\dom(b)$. Let $a\in \dom(b) \setminus B$ and consider the tournament restricted to $C\cup\{a\}$. We know from $B$'s stability that $S_\mathcal{Q}(B\cup\{a\})\subseteq B$. Furthermore, since $S_\mathcal{Q}$ satisfies strong retentiveness (\lemref{lem:dcon}), $S_\mathcal{Q}(C\cup\{a\})\subseteq C$, which shows that $C$ is $S_\mathcal{Q}$-stable in $C\cup\{a\}$. $C\cup\{a\}$ is of order $n$ or less. Hence, there is a unique minimal $S_\mathcal{Q}$-stable set in $C\cup\{a\}$, which is contained in $C$.
\end{proof}

\begin{theorem}\label{thm:teqsq}
Let $\mathcal{Q}$ be a concept of qualified subsets. Then,
\begin{enumerate}[label=\textit{(\roman*)}]
\item $\ms{S}_\mathcal{Q}$ satisfies retentiveness if $\mathcal{S}_{\mathcal{Q}}$ is directed,
\item $\mathcal{S}_{\mathcal{Q}}$ is directed if $\mathcal{R}_{\teq}$ is directed, and
\item $\teq \subseteq \ms{S}_\mathcal{Q}$ if $\mathcal{R}_{\teq}$ is directed.
\end{enumerate}
\end{theorem}
\begin{proof}
We prove all statements simultaneously by induction on the tournament order $n$. The basis is straightforward. Now, assume that all three implications hold for tournaments of order $n$ or less. If $\mathcal{S}_{\mathcal{Q}}$ is directed for tournaments of order $n+1$ or less, we can apply \lemref{lem:selfret} to show that every $S_\mathcal{Q}$-stable set is $\ms{S}_\mathcal{Q}$-retentive in such tournaments. Hence, $\ms{S}_\mathcal{Q}$ satisfies retentiveness and the first statement holds for tournaments of order $n+1$.

For the second statement, assume for contradiction that there is a tournament $T=(A,\succ)$ of order $n+1$ that contains two disjoint $S_\mathcal{Q}$-stable sets $B_1$ and $B_2$ (while $\mathcal{S}_{\mathcal{Q}}$ is directed for all tournaments of order $n$ or less and $\mathcal{R}_{\teq}$ is directed for all tournaments of order $n+1$ or less, including $T$). It follows from \lemref{lem:selfret} that $B_1$ and $B_2$ are also $\ms{S}_\mathcal{Q}$-retentive. Moreover, the induction hypothesis of the third statement implies that $\teq(\dom(a))\subseteq \ms{S}_\mathcal{Q}(\dom(a))$ for all $a\in A$, which implies that $B_1$ and $B_2$ are $\teq$-retentive, a contradiction.

In order to show the third statement, let $\mathcal{R}_{\teq}$ be directed for tournaments of order $n+1$ or less. It follows from the second implication that $\mathcal{S}_{\mathcal{Q}}$ of such tournaments is directed as well and from the first that $\ms{S}_\mathcal{Q}$ satisfies retentiveness for such tournaments. We know from \thmref{thm:teqdcon} that $\teq$ is contained in all tournament solutions that satisfy retentiveness. Hence, $\teq(T)\subseteq \ms{S}_\mathcal{Q}(T)$ for all tournaments of order $n+1$.
\end{proof}

In other words, $\tc$ and $\mc$ satisfy retentiveness and $\me$ satisfies retentiveness if \conref{con:me} holds. Furthermore, Conjecture~\ref{con:teq} is at least as strong as Conjecture~\ref{con:me}. Similarly, the directedness of $\mathcal{S}_{\mathcal{M}}$, which was proved by \citet{Dutt88a} (see \thmref{thm:Dutta}), also follows from Conjecture~\ref{con:teq}. Finally, given that \conref{con:teq}  holds, $\teq$ is a refinement of 
all tournament solutions $\ms{S}_\mathcal{Q}$ where $\mathcal{Q}$ is a concept of qualified subsets. In particular, we have the following.

\begin{corollary}\label{cor:teqsubset}
 $\teq\subseteq \me$ if $\mathcal{R}_{\teq}$ is directed
\end{corollary}

The remaining question is whether $\teq$ and $\me$ are actually different solution concepts. 
The tournament given in Figure~\ref{fig:meteq} demonstrates that this is indeed the case. 
%
\begin{figure}[htb]
  \centering\figuresize
  \begin{tikzpicture}[node distance=0.66em,vertex/.style={circle,draw,minimum size=1.6em,inner sep=0pt}]
    \node[vertex] (a1) {$a_1$};
    \foreach \x in {2,...,8} \pgfmathtruncatemacro{\y}{\x-1}
    \node[vertex,right=of a\y](a\x){$a_\x$};
	\draw [-latex] (a8) to [bend left=45] (a4);
	\draw [-latex] (a7) to [bend left=45] (a3);
	\draw [-latex] (a6) to [bend left=45] (a2);
	\draw [-latex] (a7) to [bend right=33] (a1);
	\draw [-latex] (a8) to [bend right=45] (a1);
 \end{tikzpicture}
 \caption{Example tournament $T=(A,{\succ})$ where $\me$ and $\teq$ differ ($\me[T]=A$ and $\teq[T]=A\setminus \{a_5\}$). In particular, $A\setminus \{a_5\}$ is no extending set since $a_5\in \ba[A]$ via the non-extendable transitive set $\{a_5,a_6,a_7,a_8\}$.}
 \label{fig:meteq}
\end{figure}
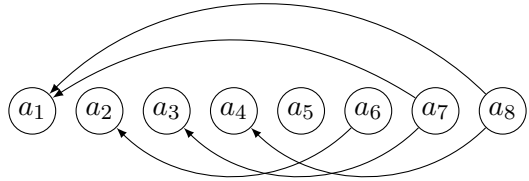

\subsection{TEQ as a Minimal Stable Set}

A natural question is whether $\teq$ itself can be represented as a minimal stable set. The following two lemmas establish that this is indeed the case if $\mathcal{R}_{\teq}$ is directed. We first show that every $S$-retentive set is $\mr{S}$-stable if $S$ satisfies \wsp for strictly smaller tournaments.

\begin{lemma}\label{lem:retstable} 
Let $S$ be a tournament solution that satisfies \wsp for all tournaments of order $n$ or less.
Then, $\mathcal{R}_S(T)\subseteq\mathcal{S}_{\mr{S}}(T)$ for all tournaments $T$ of order $n+1$.
\end{lemma}
\begin{proof}
Let $T=(A,\succ)$ be a tournament or order $n+1$ and $B$ an $S$-retentive set in $A$. If $B=A$, the statement is trivially satisfied. Otherwise, let $a\in A\setminus B$.  We first show that $B$ is $S$-retentive in $B\cup\{a\}$. 
Let $b$ be an arbitrary alternative in $B$.
$S$-retentiveness implies $S(\dom(b))\subseteq B$ and \wsp implies $S(\dom_{B\cup\{a\}}(b))\subseteq S(\dom(b))$. As a consequence, $S(\dom_{B\cup\{a\}}(b))\subseteq B$, and thus $B$ is $S$-retentive in $B\cup \{a\}$. 
It remains to be shown that $a$ is not contained in another minimal $S$-retentive subset of $B\cup \{a\}$. Assume for contradiction that $a$ is contained in some minimal $S$-retentive set. If $\{a\}$ itself were an $S$-retentive set, $a$ would be the Condorcet winner in $B\cup \{a\}$, contradicting the fact that $B$ is $S$-retentive in $B\cup \{a\}$. Now let $C\subset B\cup\{a\}$ with $a\in C$ and $|C|>1$ be a minimal $S$-retentive set. 
This implies that $B\cap C$ is also $S$-retentive, contradicting the minimality of $C$. It follows that $B$ is $\mr{S}$-stable.
\end{proof}

The next lemma shows that every every $\teq$-stable set is also $\teq$-retentive, assuming the directedness of $\mathcal{R}_{\teq}$.
\todo{This is $\widehat{\gamma}$!}
\begin{lemma}\label{lem:stableret2} 
$\mathcal{S}_{\teq}\subseteq \mathcal{R}_{\teq}$ if $\mathcal{R}_{\teq}$ is directed.
\end{lemma}
\begin{proof}
We prove the statement by induction on the tournament order $n$. We may assume that $\mathcal{R}_{\teq}$ is directed for all tournaments of order $n+1$ or less and that $\mathcal{S}_{\teq}\subseteq \mathcal{R}_{\teq}$ for all tournaments of order $n$ or less. 
Let $T=(A,\succ)$ be a tournament of order $n+1$ and $B\subseteq A$ a $\teq$-stable set in $T$. In other words, if we let $A\setminus B=\{a_1,\dots,a_k\}$, then $\teq(B\cup\{a_i\})\subseteq B$ for all $1\le i \le k$. Now, in order to show that $B$ is $\teq$-retentive, consider an arbitrary $b\in B$ and let $C=\dom_B(b)$. By definition of $\teq$, $\teq(C\cup \{a_i\})\subseteq C$ for all $a_i\in \dom_{A\setminus B}(b)$, \ie $C$ is $\teq$-stable in $\dom(b)$. Since $\mathcal{S}_{\teq}(\dom(b))\subseteq \mathcal{R}_{\teq}(\dom(b))$, $C$ is also $\teq$-retentive in $\dom(b)$ and, due to the directedness of $\mathcal{R}_{\teq}(\dom(b))$, $\teq(\dom(b))\subseteq C$. Hence, $B$ is $\teq$-retentive in $T$.
\end{proof}

We have now cleared the ground for the main result of this section.

\begin{theorem}\label{thm:teqstable}
$\mathcal{S}_{\teq}$ is directed if and only if $\mathcal{R}_{\teq}$ is directed.  $\teq=\ms{\teq}$ if $\mathcal{R}_{\teq}$ is directed.
\end{theorem}
\begin{proof}
We first prove the following two implications by induction on the tournament order $n$: 
\emph{(i)} if $\mathcal{R}_{\teq}$ is directed, then $\teq$ satisfies $\wsp$ and $\mathcal{R}_{\teq}\subseteq \mathcal{S}_{\teq}$;
\emph{(ii)} if $\mathcal{S}_{\teq}$ is directed, then $\mathcal{R}_{\teq}$ is directed. 

The basis is straightforward. Assume that both implications hold for tournaments of order $n$ and let $T=(A,\succ)$ be a tournament of order $n+1$. 
In order to prove the first statement, assume that $\mathcal{R}_{\teq}$ is directed for all tournaments of order $n+1$ or less and let $B\subseteq A$ such that $\teq(T)\subseteq B$. The induction hypothesis implies that $\teq$ satisfies $\wsp$ in all dominator sets $\dom(a)$ for $a\in A$. Hence, $\teq(T)$ is $\teq$-retentive in $B$ and, due to the directedness of $\mathcal{R}_{\teq}(T)$, $\teq(B)\subseteq \teq(T)$. Moreover, \lemref{lem:retstable} shows that $\mathcal{R}_{\teq}(T)\subseteq \mathcal{S}_{\teq}(T)$ since $\teq=\mr{\teq}$ by definition.

For the second statement, assume that $\mathcal{S}_{\teq}$ is directed for tournaments of order $n+1$ or less. It follows from the induction hypothesis that $\mathcal{R}_{\teq}$ is directed for tournaments of order $n$ or less and from the induction hypothesis of the first statement that $\teq$ satisfies $\wsp$ in these tournaments. Now, assume for contradiction that there two disjoint $\teq$-retentive sets in $T$. According to \lemref{lem:retstable}, these sets are also $\teq$-stable, which contradicts the directedness of $\mathcal{S}_{\teq}(T)$.

Finally, assume that $\mathcal{R}_{\teq}$ is directed. The first statement and \lemref{lem:stableret2} establish that $\mathcal{R}_{\teq}=\mathcal{S}_{\teq}$  and hence that $\teq=\ms{\teq}$. As a consequence, $\mathcal{S}_{\teq}$ has to be directed as well, which concludes the proof.
\end{proof}

Combining the previous theorem and the definition of $\teq$, Conjecture~\ref{con:teq} entails that 
\[\teq=\mr{\teq}=\ms{\teq}\text{.}\]

%

\begin{savequote}
\sffamily
Unser Verfahren kommt darauf hinaus, da\ss{} die relativen Spielst\"arken als Wahrscheinlichkeiten aufgefa\ss{}t und so bestimmt werden, da\ss{} die Wahrscheinlichkeit f\"ur das Eintreten des beobachteten Turnier-Ergebnisses eine m\"oglichst gro\ss{}e wird.
\qauthor{E. Zermelo, 1929}
 \end{savequote}

\section{Quantitative Concepts}\label{sec:quant}

In~\charef{sec:maxsubsets}, several solution concepts were defined by collecting the maximal elements of inclusion-maximal qualified subsets. In this \extra{chapter}\extrapaper{section}, we replace maximality with respect to set inclusion by maximality with respect to \emph{cardinality}, \ie we look at qualified subsets containing the largest number of elements. 

\subsection{Maximal Qualified Subsets}

For every set of finite sets $\mathcal{S}$, define $\max_\le (\mathcal{S})=\{S \in \mathcal{S} \mid |S|\ge |S'| \text{ for all } S'\in \mathcal{S} \}$. In analogy to Definition~\ref{def:maxqualified}, we can now define a solution concept that yields the maximal elements of the largest qualified subsets.
\begin{definition}\label{def:maxqualifiedcardinality}
Let $\mathcal{Q}$ be a concept of qualified subsets. Then, the tournament solution $S_\mathcal{Q}^\#$ is defined as
\[
S_\mathcal{Q}^\#(T) = \{ \max_\prec(B) \mid B\in \max_\le (\mathcal{Q}(T))\}\text{.}
\]
\end{definition}

Obviously, $S_\mathcal{Q}^\# \subseteq S_\mathcal{Q}$ for all concepts of qualified subsets $\mathcal{Q}$. For the concept of qualified subsets $\mathcal{M}$, \ie the set of subsets that admit a maximal element, we obtain the Copeland set. 

\paragraph{Copeland set.} $S_\mathcal{M}^\#(T)$ returns the \emph{Copeland set} $\co(T)$ of a tournament $T$, \ie the set of all alternatives whose dominion is of maximal size \citep{Cope51a}.\footnote{This set is usually attributed to Copeland despite the fact that \citet{Zerm29a} and Llull (as early as 1283, see \citealt{HaPu01a}) have suggested equivalent concepts much earlier.}
\bigskip

\subsection{Minimal Stable Sets}

When the Copeland set is taken as the basis for stable sets, some tournaments contain more than one inclusion-minimal externally stable set and, even worse, do not admit a set that satisfies both internal \emph{and} external stability (see \figref{fig:costability} for an example).

\begin{figure}[htb]
  \centering\figuresize
  \def\nd{0.66em} 
  \def\ra{0.8em} 
  \begin{tikzpicture}[node distance=\nd,vertex/.style={circle,draw,minimum size=2*\ra,inner sep=0pt}]
	\draw
    (0,0) node[circle,fill=black!15,draw=black!15,minimum size=3*\nd+4*\ra] (c) {} 
    +(90:\nd+\ra) node[vertex] (a1) {$a_1$} 
    +(330:\nd+\ra) node[vertex] (a2)  {$a_2$}
    +(210:\nd+\ra) node[vertex] (a3)  {$a_3$} ++(270:4*\ra)
    +(330:4*\ra) node[vertex] (a4) {$a_4$}
    +(210:4*\ra) node[vertex] (a5) {$a_5$};

	\draw [-latex] (a1) to (a2);
	\draw [-latex] (a2) to (a3);
	\draw [-latex] (a3) to (a1);
	\draw [-latex] (c) to (a4);
	\draw [-latex] (a4) to (a5);
	\draw [-latex] (a5) to (c);
 \end{tikzpicture}
 \caption{Example tournament $T=(A,{\succ})$ that does not contain an internally and externally $\co$-stable set. There are eight externally $\co$-stable sets (\eg $A$, $\{a_2,a_3,a_4,a_5\}$, or $\{a_1,a_2,a_5\}$), none of which is internally stable.}
 \label{fig:costability}
\end{figure}
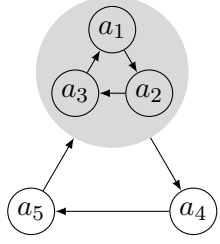

However, as it turns out, every tournament is the \emph{summary} of some tournament consisting only of homogeneous components that admits a unique internally and externally $\co$-stable set. For example, when replacing $a_4$ and $a_5$ in the tournament given in \figref{fig:costability} with 3-cycle components, the set of all alternatives is internally and externally stable.
The following definition captures this strengthened notion of stability.

\begin{definition}\label{def:scstable}
Let $S$ be a tournament solution and $\tilde{T}=(\{1,\dots,k\},\tilde{\succ})$ a tournament.  Then, $\tilde{B}\subseteq \{1,\dots,k\}$ is \emph{strongly stable} with respect to tournament solution $S$ (or \emph{strongly $S$-stable}) if there exist homogeneous tournaments $T_1,\dots,T_k$ on $k$ disjoint sets $B_1,\dots,B_k\subseteq X$ such that
$B=\bigcup_{i\in \tilde{B}} B_i$ is internally and externally $S$-stable in $T=\Pi(\tilde{T},T_1,\dots,T_k)$.
The set of strongly $S$-stable sets for a given tournament $T=(A,{\succ})$ will be denoted by $\tilde{\mathcal{S}}_S(T)=\{ B\subseteq A\mid B \text{ is strongly $S$-stable in $T$} \}$.
\end{definition}
Now, in analogy to $\ms{S}$, we can define a tournament solution that yields the minimal strongly $S$-stable set with respect to some underlying tournament solution $S$.
\begin{definition}
Let $S$ be a tournament solution. Then, the tournament solution $\mcs{S}$ is defined as
\[
\mcs{S}(T) = \bigcup \min_\subseteq(\tilde{\mathcal{S}}_S(T))\text{.}
\]
\end{definition}

The following result follows from observations  made independently by \citet{LLL93b} and \citet{FiRy95a} \citep[see also][]{Lasl97a,Lasl00a}.

\begin{theorem}[\citealp{LLL93b}]
$|\tilde{\mathcal{S}}_{S_\mathcal{M}^\#}(T)|=1$ for all tournaments $T$.
\end{theorem}

The unique strongly stable set with respect to the Copeland set is known as the \emph{bipartisan set}.

\paragraph{Bipartisan set.} The \emph{bipartisan set} of a tournament $T$ is given by $\bp(T)=\mcs{S}^\#_\mathcal{M}(T)$. It was originally defined as the set of alternatives corresponding to the support of the unique Nash equilibrium of the tournament game \citep{LLL93b}. The tournament game of a tournament is the two-player zero-sum game given by its adjacency matrix\extra{ (see \secref{sec:adversarial})}.
\bigskip

$\bp$ satisfies all basic properties, composition-consistency, and is contained in $\mc$. Its relationship with $\ba$ is unknown \citep{LLL93b}.

\begin{table}[tb]
\centering\figuresize
\begin{tabular}{ll}
\toprule
$S$ & $\mcs{S}$ \\
\midrule
Condorcet non-losers ($\cnl$) & Top cycle ($\tc$)\\
Copeland set ($\co$) & Bipartisan set ($\bp$)\\
Uncovered set ($\uc$) & Minimal covering set ($\mc$)\\
Banks set ($\ba$) & Minimal extending set ($\me$)\\
Tournament equilibrium set ($\teq$) & Tournament equilibrium set ($\teq$)\\
\bottomrule
\end{tabular}
\caption{Tournament solutions and their minimal strongly stable counterparts. The representation of $\teq$ as a stable set relies on Conjecture~\ref{con:teq}.}
\label{tbl:stronglystable}
\end{table}

Interestingly, minimality is not required for $\mcs{S}^\#_\mathcal{M}$, because there is always exactly one strongly $S^\#_\mathcal{M}$-stable set.
Further observe that internally and externally stable sets and strongly stable sets coincide when the underlying solution concept satisfies \com. This is the case for $\mc$, $\me$, and $\teq$. Even though, $\tc$ does not satisfy \com, it is easily verified that replacing alternatives with homogeneous components does not affect the top cycle of a tournament. It is thus possible to define all mentioned concepts using minimal strongly stable sets instead of stable sets (see Table~\ref{tbl:stronglystable}).

\extra{

\begin{savequote}
\sffamily
My own viewpoint is that, inter alia, a solution concept 
must be calculable, otherwise you are not going to use it.
\qauthor{R.~J.~Aumann, 1998}
\end{savequote}

\section{Computational Aspects}\label{sec:computation}

The effort required to determine the choice set of a given tournament is obviously an important property of any tournament solution.  If computing a choice set is infeasible, the applicability of the corresponding solution concept is seriously undermined. In order to make formal statements about the computational aspects of tournament solutions, we will rely on the well-established framework of \emph{computational complexity theory}~\citep[see, \eg][]{Papa94a}.  Complexity theory deals with \emph{complexity classes} of problems that are computationally equivalent in a certain well-defined way.  Typically, problems that can be solved by an algorithm whose running time is polynomial in the size of the problem instance are considered \emph{tractable}, whereas problems that do not admit such an algorithm are deemed \emph{intractable}. Formally, an algorithm is \emph{polynomial} if there exists~$k\in \mathbb{N}$ such that the running time is in~$O(n^k)$ where~$n$ is the size of the input. When $k=1$, the running time is \emph{linear}. 

For convenience, problems are usually phrased as \emph{decision problems}, \ie problems whose answer is either ``yes'' or ``no.''
The class of decision problems that can be solved in polynomial time is denoted by P, whereas NP (for ``nondeterministic polynomial time'') refers to the class of decision problems whose solutions can be \emph{verified} in polynomial time. The famous P$\ne$NP conjecture states that the hardest problems in NP do not admit polynomial-time algorithms and are thus not contained in P. Although this statement remains unproven, it is widely believed to be true.
\emph{Hardness} of a problem for a particular class intuitively means that the problem is no easier than any other problem in that class.  Both membership and hardness are established in terms of \emph{reductions} that transform instances of one problem into instances of another problem using computational means appropriate for the complexity class under consideration.  Most reductions in this chapter rely on reductions that can be computed in time polynomial in the size of the problem instances.  Finally, a problem is said to be \emph{complete} for a complexity class if it is both contained in and hard for that class.  

Given the current state of complexity theory, we cannot prove the \emph{actual} intractability of most algorithmic problems, but merely give \emph{evidence} for their intractability.  Showing NP-hardness of a problem is commonly regarded as very strong evidence for computational intractability because it relates the problem to a large class of problems for which no efficient, \ie polynomial-time, algorithm is known, despite enormous efforts to find such algorithms. 

The definition of any computable tournament solution induces a straightforward algorithm, which exhaustively enumerates all subsets of alternatives and checks which of them comply with the conditions stated in the definition. Not surprisingly, such an algorithm is very inefficient. Yet, proving the intractability of a solution concept essentially means that \emph{any} algorithm that implements this concept is asymptotically as bad as the straightforward algorithm! The following decision problem will be of central interest in this chapter:
\begin{quote}
Given a tournament solution $S$, a tournament $T=(A,{\succ})$, and an alternative $a\in A$, is $a$ contained in~$S(A)$?
\end{quote}
Deciding whether an alternative is contained in a choice set is  computationally equivalent (via so-called first-order reductions) to actually finding the set. As a consequence, we can restrict our attention to the above mentioned membership decision problem without loss of generality.

Besides P and NP, we will consider the complexity classes \aczero and \tczero. \aczero is the class of problems solvable by uniform constant-depth Boolean circuits with unbounded fan-in and a polynomial number of gates. \tczero is defined as \aczero, but additionally allows majority gates which output true if and only if the number of true inputs exceeds the number of false inputs.  Uniformity means that there is an algorithm that requires only logarithmic space for constructing the circuit~$C_n$ for each input length~$n$.  
\aczero, \tczero, P, and NP are related according to the following inclusion chain:
\[
\text{\aczero}\subset\text{\tczero}\subseteq\text{P}\subseteq\text{NP}\text{.}
\]
The inclusion relationships between \tczero and P and between P and NP are believed to be strict. 

Finite model theory has revealed a strong relationship between complexity theory and logic in a finite, ordered universe \citep[see, \eg][]{Imme98a,Libk04a}. This underlines the robustness of common complexity classes as their logical characterizations do not invoke machine models or notions such as time and space, let alone ``polynomial'' and ``exponential.'' NP, for instance, is the class of predicates definable in existential second-order logic. We will make no appeal to these logical characterizations except when showing membership in \aczero as this class simply corresponds to first-order logic.\footnote{More precisely, uniform \aczero corresponds to first-order logic with a so-called BIT predicate. Our results, however, do not require the BIT predicate.}
For a binary operator $\succ$ representing the dominance relation, a first-order theory of tournaments is given by the axioms 
\[\forall x\, \forall y\, (\neg x=y \leftrightarrow x\succ y \vee y\succ x)
\quad\text{and}\quad
\forall x\, \forall y (x\succ y \rightarrow \neg y\succ x)\text{.}\]
The membership decision problem for tournament solution $S$ is in \aczero if there is a first-order expression that is true for a given set $A$ and a binary relation $\succ$ if and only if $(A,\succ)$ is a tournament and $a\in S(A)$. In order to simplify first-order expressions, we furthermore let $x\succeq^0 y\leftrightarrow x=y$ and $x\succeq^{k+1}y \leftrightarrow x\succeq^k y \vee \exists z\, (x\succeq^k z\wedge z\succ y)$.

In the remainder of this chapter, we will analyze the computational complexity of the most common tournament solutions of our framework, namely the Copeland set, the uncovered set, the Banks set, the top cycle, the bipartisan set, the minimal covering set, the minimal extending set, and the tournament equilibrium set.

\subsection{Copeland Set}

The Copeland set consists of all alternatives with maximal degree, \ie all alternatives whose dominion is of maximal size. Not surprisingly, this set can be easily computed in linear time (see Algorithm~\ref{alg:co}).

\begin{algorithm}[htb]
\begin{algorithmic}
	\STATE \hspace*{-1em} \textbf{procedure} $\co(A,{\succ})$
    \STATE $B\leftarrow \{ a\in A\mid \forall b\in A\colon |D(a)|\ge |D(b)\}$
	\RETURN $B$
\end{algorithmic}
\caption{Copeland set}
\label{alg:co}
\end{algorithm}

However, the membership decision problem for the Copeland set is not in \aczero and therefore not expressible in first-order logic.

\begin{theorem}[\citealp{BFH09b}]\label{thm:co_complexity}
	Deciding whether an alternative is contained in the Copeland set of a tournament is \tczero-complete under first-order reductions.
\end{theorem}

\subsection{Uncovered Set}
\label{sec:uc_complexity}

The uncovered set is usually defined via a subrelation of the dominance relation called \emph{covering relation} rather than in terms of maximal sets that admit a maximal element (\cf \secref{sec:maxmaxqualified}).

\begin{definition}\label{def:coveringrelation}
	Let~$(A,{\succ})$ be a tournament.  Then, for every $a,b\in A$, $a$ \emph{covers}~$b$, in symbols $a \cov b$, if $D(a)\supset D(b)$.
\end{definition}
Observe that $a\supset b$ implies $a\succ b$ and that $a\supset b$ if and only if $\dom(a)\subset \dom(b)$.
It is easily verified that the covering relation is a transitive subrelation of the dominance relation.  The set of maximal elements of the covering relation of a given tournament $T$ is the uncovered set $\uc(T)=\{a\in A \mid \text{$b \cov a$ for \emph{no} $b\in A$}\}$ \citep{Fish77a,Mill80a}. 

Interestingly, the uncovered set consists precisely of those alternatives that reach every other alternative on a domination path of length at most two. In graph theory, these alternatives are often called the \emph{kings} of a tournament.

\begin{theorem}[\citealp{ShWe84a}]
$\uc(T)=\{a\in A\mid D_{\succeq^2}(a)=A\}$ for all tournaments $T=(A,{\succ})$.
\end{theorem}
\begin{proof}
For the direction from left to right, let $a\in \uc(T)$. Clearly, if $b\in D(a)$, then $b\in D_{\succeq^2}(a)$. Let therefore $b\in \dom(a)$. Since $a$ lies within the uncovered set, $b$ may not cover $a$, \ie $\dom(b)\not\subseteq\dom(a)$. Thus, there has to be a $c\in \dom(b)$ such that $c \not\succ a$. As a consequence, $a\succ c \succ b$ and $b\in D_{\succeq^2}(a)$.

For the converse, let $a\not\in \uc(T)$. Then, there has to be $b\in \dom(a)$ such that $b \cov a$, which implies that $b$ dominates all $c$ in $D(a)$ and consequently that $b\not\in D_{\succeq^2}(a)$. 
\end{proof}

The previous characterization can be leveraged to compute the uncovered set via matrix multiplication \citep{Land53a}. Using the fastest known algorithm for matrix multiplication \citep{CoWi90a}, the asymptotic complexity of Algorithm~\ref{alg:uc} is $O(n^{2.38})$ \citep[see, also][]{Hudr09a}.\footnote{There is some evidence that $O(n^2)$ algorithms for matrix multiplication exist \citep{CKSU05a}.}

\begin{algorithm}[htb]
\begin{algorithmic}
	\STATE \hspace*{-1em} \textbf{procedure} $\uc(A,{\succ})$
    \FORALL{$i,j\in A$}
	\STATE \textbf{if} $i\succ j \vee i=j$ \textbf{then} $m_{ij}\leftarrow 1$
	\STATE \textbf{else} $m_{ij}\leftarrow 0$ \textbf{end if}
	\ENDFOR
	\STATE $M\leftarrow (m_{ij})_{i,j\in A}$
    \STATE $U\leftarrow (u_{ij})_{i,j\in A} \leftarrow M^2+M$
    \STATE $B\leftarrow \{ i\in A\mid \forall j\in A\colon u_{ij}\ne 0\}$
	\RETURN $B$
\end{algorithmic}
\caption{Uncovered set}
\label{alg:uc}
\end{algorithm}

Since $\co\subseteq \uc$, the alternatives that dominate most alternatives dominate \emph{every} alternative on a path of length at most two. From a complexity theory point of view, deciding whether an alternative lies within $\uc(T)$ is computationally easier than checking whether it is contained in $\co(T)$, despite the fact that the fastest known algorithm for computing $\uc$ is asymptotically slower than the fastest algorithm for $\co$. Computing the uncovered sets can be highly parallelized by determining the covering relation for every pair of alternatives.

\begin{theorem}[\citealp{BrFi08b}]\label{thm:uc_complexity}
	Deciding whether an alternative is contained in the uncovered set of a tournament is in \aczero.
\end{theorem}

The previous theorem follows straightforwardly from the fact that there is a simple first-order expression for checking membership in the uncovered set,   
\[
\uc(x) \;\leftrightarrow\; \forall y\, (x\succeq^2 y)\text{.}
\]

The theorem also implies that the iterated uncovered set $\uc^\infty$ can be computed in polynomial time.


\subsection{Banks Set}
\label{sec:ba_complexity}

In contrast to the previous two tournament solutions, the Banks set, which consists of the maximal alternatives of maximal transitive subsets, cannot be computed in polynomial time unless P equals NP.

\begin{theorem}[\citealp{Woeg03a}]\label{thm:ba_complexity}
	Deciding whether an alternative is contained in the Banks set of a tournament is NP-complete.
\end{theorem}

An alternative, arguably simpler, proof of this statement has been given by \citet{BFHM09a}.
Interestingly, random alternatives (and thus subsets) of the Banks set can be found in linear time \citep{Hudr04a}. This is can be achieved by iteratively assembling maximal transitive sets (see Algorithm~\ref{alg:ba_element}). The difficulty of computing the Banks sets is rooted in the potentially exponential number of maximal transitive subsets contained in a tournament.

\begin{algorithm}[htb]
\begin{algorithmic}
	\STATE \hspace*{-1em} \textbf{procedure} $\mathit{BA-ELEMENT}(A,{\succ})$
    \STATE $B\leftarrow \emptyset$
    \STATE $a\in A$
    \LOOP
    \STATE $B\leftarrow B\cup \{a\}$
    \STATE $C\leftarrow \{c \mid c\succ B\}$
    \STATE \textbf{if} $C=\emptyset$ \textbf{then return} $a$ \textbf{end if}
    \STATE $a\in C$
    \ENDLOOP
\end{algorithmic}
\caption{Banks set element}
\label{alg:ba_element}
\end{algorithm}

\subsection{Top Cycle}

The top cycle or minimal dominant set of a tournament $T$ can be computed in linear time by starting with an arbitrary non-empty (linear-time computable) set of alternatives contained in $\tc(T)$, \eg $\co(T)$, and then iteratively adding alternatives that are not dominated by the current set (see Algorithm~\ref{alg:tc}).\footnote{Slightly more efficient algorithms can be obtained by relying on the fact that the degree of any alternative in the top cycle is always at least as high as the degree of any alternative outside the top cycle \citep{Moon68a}.}

\begin{algorithm}[htb]
\begin{algorithmic}
	\STATE \hspace*{-1em} \textbf{procedure} $\tc(A,{\succ})$
	\STATE $B\leftarrow C\leftarrow \co(A,{\succ})$
    \LOOP
    \STATE $C\leftarrow \bigcup_{a\in C} \dom_{A\setminus B}(a)$
    \STATE \textbf{if} $C=\emptyset$ \textbf{then return} $B$ \textbf{end if}
    \STATE $B\leftarrow B\cup C$
    \ENDLOOP
\end{algorithmic}
\caption{Top Cycle}
\label{alg:tc}
\end{algorithm}

The following lemma will prove useful for characterizing the computational complexity of the top cycle.

\begin{lemma}\label{lem:smith_bilateral}
$\tc(T)=\{a\in A\mid D_{\succeq^*}(a)=A\}$ for all tournaments $T=(A,{\succ})$.
\end{lemma}
\begin{proof}
Let $T=(A,\succ)$ be a tournament. 
For the inclusion from left to right, let $a\in \tc(T)$.
We first show that for every $b\in A$, $\overline{D}{}_{\succeq^*}(b)$ is a dominant set. Assume for contradiction that there are alternatives $c\in \overline{D}{}_{\succeq^*}(b)$ and $d\not\in \overline{D}{}_{\succeq^*}(b)$ such that $c\not\succ d$. This implies that $d\succ c$ and consequently that $d\in \overline{D}{}_{\succeq^*}(b)$, a contradiction. Since $a$ is contained in all dominant sets, $a\in \overline{D}{}_{\succeq^*}(b)$ for every $b\in A$. As a consequence, $D_{\succeq^*}(a)=A$.

For the converse, assume that $a\not\in \tc(T)$ and let $b\in \tc(T)$. Since all alternatives in $\tc(T)$ dominate all alternatives outside $\tc(T)$, no alternative in $\tc(T)$, including $b$, can be contained in $D_{\succeq^*}(a)$.
\end{proof}

\citet{NiTa05a} have shown that reachability in tournaments can be decided in \aczero. This implies that deciding membership in the top cycle is also in \aczero. 

\begin{theorem}[\citealp{BFH09b}]\label{thm:tc_complexity}
Deciding whether an alternative is contained in the top cycle of a tournament is in \aczero.
\end{theorem}

Interestingly, \citeauthor{NiTa05a}'s proof relies on the existence of an alternative that dominates every other alternative in at most two steps (see \secref{sec:uc_complexity}). As in the case of the uncovered set, there is a first-order expression for membership in the top cycle,
\[
\tc(x) \;\leftrightarrow\; \forall y\,\forall z\, (\forall v\, (z\succeq^3 v \rightarrow z\succeq^2 v)\wedge z\succeq^2 x \rightarrow z\succeq^2 y)\text{.}
\]

\subsection{Bipartisan Set}
\label{sec:bp_complexity}

The bipartisan set, \ie the set of actions of the tournament game that are played with positive probability in the unique Nash equilibrium (see also \secref{sec:adversarial}), can be found efficiently by solving a linear feasibility problem (see Algorithm~\ref{alg:bp}).

\begin{algorithm}[htb]
\begin{algorithmic}
	\STATE \hspace*{-1em} \textbf{procedure} $\bp(A,{\succ})$
    \FORALL{$i,j\in A$}
	\STATE \textbf{if} $i\succ j$ \textbf{then} $m_{ij}\leftarrow 1$
	\STATE \textbf{else if} $j\succ i$ \textbf{then} $m_{ij}\leftarrow -1$ 
	\STATE \textbf{else} $m_{ij}\leftarrow 0$ \textbf{end if}
	\ENDFOR
	\STATE $s \in \{s\in \mathbb{R}^n \mid $
	$\begin{array}[t]{ll}
		\sum_{j\in A} s_j\cdot m_{ij}\leq 0 & \forall i\in A \\
 		\sum_{j\in A} s_j = 1 \\
		s_j\geq 0 & \forall j\in A\}\\
    \end{array}$
	\STATE $B\leftarrow \{\, a\in A\mid s_a>0\,\}$
	\RETURN $B$
\end{algorithmic}
\caption{Bipartisan set}
\label{alg:bp}
\end{algorithm}

\begin{theorem}[\citealp{BrFi08b}]\label{thm:bp_complexity}
Deciding whether an alternative is contained in the bipartisan set is in P.
\end{theorem}
In incomplete tournaments---a generalization of tournaments where the dominance relation is not required to be complete---the membership decision problem is P-\emph{complete} \citep{BrFi08b}. Whether this also holds for tournaments is an open problem.

\subsection{Minimal Covering Set}

Since its introduction in 1988, there has been doubt whether the minimal  covering set can be computed efficiently \citep{Dutt88a,Lasl97a}. In contrast to all solution concepts considered so far in this chapter, there is no obvious reason why the corresponding decision problem should be in NP, \ie even \emph{verifying} whether a given set is indeed a minimal covering set is a non-trivial task. While it can easily be checked whether a set is a covering set, verification of minimality is problematic. For all previously considered concepts, there are witnesses of polynomial size---maximal sets that admit a maximal element in the case of the uncovered set or maximal transitive sets in the case of the Banks set---that permit the efficient verification of a choice set.
The problem of verifying a minimal covering set and the more general problem of deciding whether a given alternative $a$ is contained in $\mc$ are both in coNP, the class consisting of all decision problems whose complement is in NP. This is due to the fact that $\mc$ is contained in all covering sets.  In other words, $a\not\in \mc$ if and only if there is a (not necessarily minimal) covering set $B\subseteq A$ with $a\not\in B$.

By the inclusion of $\bp$ in $\mc$ \citep{LLL93b}, Algorithm~\ref{alg:bp} provides a way to efficiently compute a \emph{non-empty subset} of $\mc$.  While in general the existence of an efficiently computable subset cannot be exploited to efficiently compute the set itself (just recall the case of the Banks set from \secref{sec:ba_complexity}),
it is of great benefit in the case of $\mc$.

Algorithm~\ref{alg:mc} resembles Algorithm~\ref{alg:tc} in that it starts with an efficiently computable subset of the choice set to be computed---in this case the bipartisan set---and then iteratively adds specific elements outside the current set that are still uncovered.  However, in contrast to Algorithm~\ref{alg:tc}, we must only add elements that may not be covered in a later iteration, and it is unclear which elements these should be.  \lemref{lem:mcalgo} shows that elements in the minimal covering set of the uncovered alternatives can be safely added to the current set.  Since a non-empty subset of any minimal covering set, viz. the bipartisan set, can be found efficiently, this completes the algorithm.

\begin{algorithm}[htb]
\begin{algorithmic}
	\STATE \hspace*{-1em} \textbf{procedure} $\mc(A,{\succ})$
	\STATE $B\leftarrow \bp(A,{\succ})$	
	\LOOP
	\STATE $A'\leftarrow \bigcup_{a\in A\setminus B} (\uc(B\cup\{a\})\cap \{a\})$
	\STATE \textbf{if} $A'=\emptyset$ \textbf{then return} $B$ \textbf{end if}
	\STATE $B\leftarrow B\cup \bp(A',{\succ})$
	\ENDLOOP
\end{algorithmic}
\caption{Minimal covering set}
\label{alg:mc}
\end{algorithm}

%
\begin{lemma}\label{lem:mcalgo}
Let~$T=(A,{\succ})$ be a tournament, $B\subseteq\mc(T)$, and $A'=\bigcup_{a\in A\setminus B}(\uc(B\cup\{a\})\cap \{a\})$.  Then, $\mc(A')\subseteq\mc(T)$.
\end{lemma}
\begin{proof}
	We extend the notation introduced in \defref{def:coveringrelation} by writing $a\cov_Zb$ if $D_Z(b)\subset D_Z(a)$ for an arbitrary set $Z\subseteq A$.
	Now, partition~$A'$, the set of alternatives not covered by $B$, into two sets~$C$ and~$C'$ of elements contained in $\mc(T)$ and elements not contained in $\mc(T)$, \ie $C=A'\cap\mc(T)$ and $C'=A'\setminus\mc(T)$.  We will show that~$C$ is externally stable within~$A'$.  Since $\mc(A')$ is contained in all sets that are externally stable within~$A'$, this means that $\mc(A')\subseteq\mc(T)$.

	In the following, we will use an easy consequence of the definition of the  covering relation: for two sets $Z,Z'$ with $Z'\subseteq Z\subseteq A$ and two alternatives $a,b\in Z'$, if~$a \cov_{Z} b$, then~$a \cov_{Z'} b$.  We will refer to this property as ``covering in subsets.''

	Let $a\in C'$.  Since $a\notin\mc(T)$, there has to be some $b\in\mc(T)$ such that $b \cov_{\mc(T)\cup\{a\}} a$.  It is easy to see that $b\notin B$.  Otherwise, since $B\subseteq\mc(T)$ and by covering in subsets, $b \cov_{B\cup\{a\}} a$, contradicting the assumption that~$a\in A'$.
	On the other hand, assume that $b\in\mc(T)\setminus(B\cup C)$.  Since~$b \cov_{\mc(T)\cup\{a\}} x$, $B\subseteq\mc(T)$, and by covering in subsets, we have that $D_B(a)\subset D_B(b)$.  Furthermore, since $b\notin A'$, there has to be some $c\in B$ such that $D_B(b)\subset D_B(c)$.  Combining the two, we have $D_B(a)\subset D_B(c)$, \ie~$c \cov_{B\cup\{a\}} a$.  This again contradicts the assumption that $a\in A'$. 
	It thus has to be the case that $b\in C$.  Since $C\subseteq\mc(T)$, it follows from covering in subsets that~$b \cov_{C\cup\{a\}} a$.  We have shown that for every $a\in C'$, there exists $b\in C$ such that $b \cov_{C\cup\{a\}} a$, \ie~$C$ is externally stable within~$A'$.
\end{proof}

Since $B$ and $A'$ in the statement of \lemref{lem:mcalgo} are always disjoint, we obtain a powerful tool: For every proper subset of the minimal  covering set, the lemma tells us how to find another disjoint and non-empty subset. This tool allows us to iteratively compute $\mc$ (see Algorithm~\ref{alg:mc}).\footnote{Lemma~\ref{lem:mcalgo} can also be used to construct a \emph{recursive} algorithm for computing $\mc$ without making any reference to $\bp$. However, such an algorithm has exponential worst-case running time.}

\begin{theorem}[\citealp{BrFi08b}]\label{thm:mc_complexity}
Deciding whether an alternative is contained in the minimal covering set of a tournament is in P.
\end{theorem}

\subsection{Tournament Equilibrium Set}

As in the case of the minimal covering set, there has been concern whether the tournament equilibrium set of a tournament can be computed efficiently \citep{Lasl97a}.  In contrast to the minimal covering set, it turned out that there exists no efficient algorithm for this task unless P equals NP.

\begin{theorem}[\citealp{BFHM09a}]\label{thm:teq_complexity}
Deciding whether an alternative is contained in the tournament equilibrium set of a tournament is NP-hard.
\end{theorem}

The proof the previous theorem actually shows that the membership decision problem for \emph{any} tournament solution that is sandwiched between $\ba$ and $\teq$, \ie any tournament solution $S$ with $\teq\subseteq S\subseteq \ba$, is NP-hard. We thus obtain the following corollary.

\begin{corollary}\label{cor:me_complexity}
Deciding whether an alternative is contained in the minimal extending set of a tournament is NP-hard if $\mathcal{R}_{\teq}$ is directed.
\end{corollary}
There are no obvious reasons why any of these problems---membership in $\teq$ and $\me$---should be in NP. They are possibly much harder.

\thmref{thm:teq_complexity} implies that the existence of a polynomial-time algorithm for computing $\teq$ is highly unlikely. Nevertheless, a practical algorithm for $\teq$ that runs reasonably well on moderately-sized instances, even though its worst-case complexity is of course exponential, would be very useful. The following alternative definition of $\teq$ will be valuable for devising a relatively efficient $\teq$ algorithm. Let $A$ be a set and $R$ a binary relation on $A$. Then, $\mtc(A,R)$ is defined as the set consisting of the maximal elements of the asymmetric part of the transitive closure of $R$ in $A$. Now, define the tournament equilibrium set of a tournament $(A,{\succ})$ as $\teq[A,\succ]=\mtc(A,\teqrel[A])$ where $\teqrel[B]$ is defined as the binary relation on $B\subseteq A$ such that, for all $a,b\in B$,  $a\teqrel[B]b$ if and only if $a\in\teq[{\dom_B(b)}]$. Further let $\teq[\emptyset,{\succ}]=\emptyset$.  
Like the covering relation, the $\teq$ relation~$\teqrel[A]$ is a subrelation of the dominance relation~$\succ$.  
A naive algorithm for computing $\teq$ can be implemented by recursively computing the $\teq$ relation and then returning the maximal elements of the asymmetric part of its transitive closure, which can be efficiently computed using standard algorithms such as that of \citet{Tarj72a}. The performance of this algorithm can be significantly improved by assuming that Conjecture~\ref{con:teq} holds \citep{BFHM09a}. This assumption can fairly be made because otherwise $\teq$ would not satisfy any of the basic properties and thus be severely compromised as a solution concept (see \secref{sec:teqsection}). Consequently, the issue of computing $\teq$ would be irrelevant.

\begin{algorithm}[htb]
\begin{algorithmic}
	\STATE \textbf{procedure} $\teq(A,{\succ})$
	\STATE $R\leftarrow \emptyset$
	\STATE $B\leftarrow C\leftarrow \co(A,{\succ})$	
	\LOOP
	\STATE $R\leftarrow R\cup \{ (b,a) \colon a\in C \wedge b\in \teq(\dom(a))\}$
	\STATE $D\leftarrow \bigcup_{a\in C} \teq(\dom(a))$
	\STATE \textbf{if} $D\subseteq B$ \textbf{then return} $\mtc(B,R)$ \textbf{end if}
	\STATE $C\leftarrow D$
	\STATE $B\leftarrow B\cup C$
	\ENDLOOP
\end{algorithmic}
\caption{Tournament Equilibrium Set}
\label{alg:teq}
\end{algorithm}

Algorithm~\ref{alg:teq} computes $\teq(T)$ by initializing working set $B$ with the Copeland set $\co(T)$. These alternatives are good candidates to be included in $\teq(T)$ and the small size of their dominator sets speeds up the computation of their $\teq$-dominators. Then, all alternatives that $\teq$-dominate any alternative in $B$ are iteratively added to $B$ until no more such alternatives can be found, in which case the algorithm returns $\mtc$ of~$\teqrel[B]$.

Experimental results show that Algorithm~\ref{alg:teq} dramatically outperforms a naive algorithm that computes the entire $\teq$ relation \citep{BFHM09a}.

\subsection{Summary}

\tabref{tbl:complexity} summarizes the computational complexity of the considered tournament solutions. There are linear-time algorithms for the Copeland set and the top cycle. Moreover, a single element of the Banks set can be found in linear time. Computing the Banks set, the tournament equilibrium set, and the minimal extending set is intractable unless P equals NP. Apparently, the minimal covering set and the bipartisan set fare particularly well in terms of the tradeoff between efficient computability and choice-theoretic attractiveness (see also \tabref{tbl:comparison}).

\begin{table}[htb]
\centering\figuresize
\begin{tabular}{lrr}
\toprule
Tournament solution & Computational complexity\\
\midrule
Copeland set ($\co$) & \tczero-complete\\
Uncovered set ($\uc$) & in \aczero \\
Banks set ($\ba$) & NP-complete \\
Top cycle ($\tc$) & in \aczero \\
Bipartisan set ($\bp$) & in P \\
Minimal covering set ($\mc$) & in P \\
Minimal extending set ($\me$) & NP-hard \\
Tournament equilibrium set ($\teq$) & NP-hard \\
\bottomrule
\end{tabular}
\caption{The computational complexity of tournament solutions}
\label{tbl:complexity}
\end{table}


\begin{savequote}
\sffamily
Although the motivation and most of the applications of this study are from animal societies it should be remarked that there are other examples of dominance relations.
\qauthor{H.~G.~Landau, 1951}
\end{savequote}

\section{Applications to Decision-Making}\label{sec:applications}

Unsurprisingly, a framework as general as that of tournament solutions has numerous applications spanning diverse areas, among them \emph{sports competitions} \citep[see, \eg][]{Usha76a,MGGM05a}, \emph{webpage and journal ranking} \citep[see, \eg][]{KoSt07a,BrFi07b}, \emph{biology} \citep[see, \eg][]{Land51a,Land51b,Land53a}, and \emph{psychology} \citep[see, \eg][]{Schj22a,Slat61a}. 
Here, we will focus on applications related to decision-making, namely \emph{collective decision-making} (social choice theory), \emph{adversarial decision-making} (theory of zero-sum games), and \emph{coalitional decision-making} (cooperative game theory).\footnote{Besides the given examples of \emph{multi-agent} decision making, tournament solutions have also been applied to \emph{individual} decision making \citep[see, \eg][]{BMP+06a}.} It is quite remarkable that the essence of each of these loosely related areas can be boiled down to choice based on pairwise comparisons.

\subsection{Collective Decision-Making}
\todo{check $(R,A)$ and $(A,R)$}
Many problems in the social sciences involve a group of agents that ought to choose from a set of alternatives in a way that is faithful to the agents' individual preferences over the alternatives. Such problems of preference aggregation are commonly studied in \emph{social choice theory} and have various applications in political science, welfare economics, and computational multiagent systems.


We consider a finite set of agents~$N=\{1,\dots,n\}$ who entertain preferences over a universal set of alternatives or candidates~$X$. Each agent~$i\in N$ possesses a transitive and complete preference relation~$R_i$ over the alternatives in~$X$.
We have~$a \mathrel{R_i} b$ denote that agent~$i$ values alternative~$a$ at least as much as alternative~$b$. In accordance with the conventional notation, we write~$P_i$ for the strict part of~$R_i$, \ie~$a \mathrel{P_i} b$ if~$a \mathrel{R_i} b$ but not~$b \mathrel{R_i} a$. 
Similarly,~$I_i$ denotes~$i$'s indifference relation, \ie $a \mathrel{I_i} b$ if both~$a \mathrel{R_i} b$ and~$b \mathrel{R_i} a$. 
A restricted type of preference relations are \emph{linear orders}---\ie  transitive, complete, and anti-symmetric relations---over the alternatives.
The set of all preference relations over a set of alternatives $X$ will be denoted by~$\mathcal{R}(X)\subseteq X\times X$. 


The central object of study in this section are \emph{social choice functions}, \ie functions that map the individual preferences of the agents and a finite set of feasible alternatives to a set of socially preferred alternatives.

\begin{definition}
A \emph{social choice function (SCF)} is a function $f:\fone(X)\times \mathcal{R}(X)^N \rightarrow \fone(X)$ such that $f(A,R)\subseteq A$ for all $R\in \mathcal{R}(X)^N$ and $A\in \fone(X)$.
\end{definition}

We will now define three conditions on SCFs, each of which highlights a different aspect of reasonable preference aggregation.
The first condition is \emph{Pareto-optimality}, which states that an alternative should not be chosen if there exists another alternative that \emph{all} agents unanimously prefer to the former.

\begin{definition}
An SCF $f$ satisfies \emph{Pareto-optimality} if for all $A\in \fone(X)$ and $R=(R_1,\dots,R_n)\in \mathcal{R}(X)^N$, $a\in A\setminus f(A,R)$ implies that there exists $b\in A$ with $b \mathrel{P_i} a$ for all $i\in N$.
\end{definition}

The next condition requires that choices from a set of feasible alternatives should not depend on preferences over alternatives that are not contained in this set.

\begin{definition}
An SCF $f$ satisfies \emph{independence of irrelevant alternatives (IIA)} if $f(A,R)=f(A,R')$ for all $A\in \fone(X)$ and $R,R'\in \mathcal{R}(X)^N$ such that $R|_A=R'|_{A}$.
\end{definition}

Finally, the third---and most demanding---condition demands that choices from a set and its subsets should be consistent.

\begin{definition}
An SCF $f$ satisfies the \emph{weak axiom of revealed preference (WARP)} if $f(B,R)=B\cap f(A,R)$ for all $A,B\in \fone(X)$ such that $B\subseteq A$ and $B\cap f(A,R)\ne \emptyset$.
\end{definition}

WARP has been shown equivalent to the existence of a complete and transitive relation on the universal set of alternatives, the maximal elements of which are precisely those that are chosen from any feasible set of alternatives \citep{Arro59a}. The existence of such a relation is usually equated with the \emph{rationality of choice}.

A minimal requirement for every SCF is that it should be sensitive to the preferences of more than one agent. In particular, there should \emph{not} be a single agent who can rule out alternatives no matter which preferences the other agents have. Such an agent is usually called a dictator.

\begin{definition}
An SCF $f$ is \emph{dictatorial} if there exists $i\in N$ such that for all $A\in \fone(X)$, $a,b\in A$, and $R=(R_1,\dots,R_n)\in \mathcal{R}(X)^N$, if  $b \mathrel{P_i} a$, then $a\not\in f(A,R)$.
\end{definition}

Now, \citeauthor{Arro51a}'s famous impossibility theorem states that there exists no non-dictatorial SCF that meets all three seemingly mild conditions defined above.

\begin{theorem}[\citealp{Arro51a,Arro59a}]
Every SCF that satisfies Pareto-optimality, IIA, and WARP is dictatorial.
\end{theorem}


Put in other words, the four properties---non-dictatorship, Pareto-optimality, IIA, and WARP---are logically inconsistent and (at least) one of them needs to be excluded to obtain positive results. Clearly, dictatorship is not acceptable and IIA merely states that the SCF represents a reasonable model of preference aggregation \citep[see, \eg][]{Schw86a,BoTi91a}. \citet{Wils72a} has shown that dropping Pareto-optimality only allows SCFs that are constant (\ie completely unresponsive) or fully determined by the preferences of a single agent. Thus, the only remaining possibility is to exclude WARP.

At this point, we will pursue two different paths. The first one is to weaken WARP by substituting it with less prohibitive conditions whereas the second one ignores consistency with respect to the set of alternatives altogether and instead focusses on consistency with respect to the set of voters.

\subsubsection{Variable Set of Alternatives}\label{sec:varalt}

Before addressing WARP, let us add a small number of intuitively acceptable conditions, which yield a one-to-one correspondence between SCFs and tournament solutions. The first two conditions simply require the SCF to be symmetric with respect to the agents and the alternatives, \ie agents and alternatives are treated equally. These conditions are strengthenings of non-dictatorship and IIA, respectively. 

\begin{definition}
An SCF $f$ is \emph{anonymous} if $f(A,R)=f(A,R')$ for all $A\in \fone(X)$,  $R=(R_1,\dots,R_n),R'=(R'_1,\dots,R'_n)\in \mathcal{R}(X)^N$, and permutations $\pi:N\rightarrow N$ such that $R'_i=R_{\pi(i)}$ for all $i\in N$.
\end{definition}

Clearly, a dictatorial SCF cannot be anonymous.

\begin{definition}
An SCF $f$ is \emph{neutral} if $\pi(f(A,R))=f(A,R')$ for all $A\in \fone(X)$, $R=(R_1,\dots,R_n),R'=(R'_1,\dots,R'_n)\in \mathcal{R}(X)^N$, and permutations $\pi:A\rightarrow A$ such that $a \mathrel{R'_i} b$ if and only if $\pi(a)\mathrel{R_i}\pi(b)$ for all $a,b\in X$ and $i\in N$.
\end{definition}

Neutrality can be seen to imply IIA by letting $\pi$ be the identity function in the definition above.

It also appears to be reasonable to demand that SCFs are monotonic in the sense that increased support may not hurt an alternative in feasible sets that only consist of two alternatives.

\todo{consider definition from set-rationalizable choice paper, this definition implies IIA on pairs!}
\begin{definition}
An SCF $f$ is \emph{positive responsive} if $f(\{a,b\},R')=\{a\}$ for all $a,b\in X$, $R=(R_1,\dots,R_n),R'=(R'_1,\dots,R'_n)\in \mathcal{R}(X)^N$ such that there exists $i\in N$ with $R_j=R'_j$ for all $j\ne i$ and
\begin{enumerate}[label=\textit{(\roman*)}]
\item $f(\{a,b\},R)=\{a\}$, $b \mathrel{R_i} a$, and $a \mathrel{R'_i} b$, or
\item $f(\{a,b\},R)=\{a,b\}$, $a \mathrel{I_i} b$ and $a \mathrel{P'_i} b$, or $b \mathrel{P_i} a$ and $a \mathrel{R'_i} b$.
\end{enumerate}
\end{definition}

\citet{May52a} has shown that in the case of only two alternatives, anonymity, neutrality, and positive responsiveness completely characterize \emph{majority rule}, \ie the voting rule that chooses an alternative whenever at least half of the voters prefer it to the other alternative. Now, given a preference profile $R\in \mathcal{R}(X)$ and a set of feasible alternatives $A\in \fone(X)$, define a dominance relation such that alternative~$a$ dominates alternative~$b$ whenever more than half of the voters prefer $a$ to $b$. Obviously, this dominance relation is asymmetric and guaranteed to be complete, \ie a tournament, if the number of agents is odd and individual preferences are linear.\footnote{The reasonable, but relatively strong, assumption of so-called \emph{single-peaked} individual preferences guarantees that the dominance relation is transitive \citep{Blac48a,Inad69a}.} Moreover, \citet{McGa53a} has shown that \emph{any} such dominance relation can be obtained under the given conditions.\footnote{Improving on McGarvey's result, \citet{Stea59a} has shown that any tournament can be realized via the simple majority rule when the number of voters is at least two greater than the number of alternatives.}
The remaining condition for our characterization is a strengthening of IIA that prescribes that all choices must only depend on the dominance relation or, more generally, on pairwise choices.

\todo{this should be called binariness}
\begin{definition}
An SCF $f$ is \emph{based on pairwise choices} if $f(A,R)=f(A,R')$ for all $A\in \fone(X)$, $R,R'\in \mathcal{R}(X)^N$, and $a,b\in A$ such that $f(\{a,b\},R)=f(\{a,b\},R')$.
\end{definition}

We thus arrive at the following correspondence between tournament solutions and SCFs.

\begin{quote}
Let the number of agents be odd and all individual preference relations linear.
Then every SCF that is based on pairwise choices and satisfies neutrality, anonymity, and positive responsiveness corresponds to a tournament solution and vice versa.
\end{quote}

Consider for example the preference relations of the three voters given in \figref{fig:ccycle}. A majority of voters (two out of three) prefers $a$ to $b$. Another one prefers $b$ to $c$ and yet another one $c$ to $a$.

\begin{figure}[htb]
  \centering\figuresize
 $\begin{array}{ccc}
  P_1 & P_2 & P_3\\\midrule
  a & b & c\\
  b & c & a\\
  c & a & b\\
  \end{array}$\qquad\qquad  
  \def\nd{0.88em} 
  \def\ra{0.7em} 
  \begin{tikzpicture}[baseline=(current bounding box.center),node distance=\nd,vertex/.style={circle,draw,minimum size=2*\ra,inner sep=0pt}]
	\draw
    +(90:\nd+\ra) node[vertex] (a) {$a$} 
    +(210:\nd+\ra) node[vertex] (c)  {$c$}
    +(330:\nd+\ra) node[vertex] (b)  {$b$};
	\draw [-latex] (a) to (b);
	\draw [-latex] (b) to (c);
	\draw [-latex] (c) to (a);
 \end{tikzpicture}
 \caption{Condorcet's paradox \citep{Cond85a}. The left-hand side shows the individual preferences of three agents such that the pairwise majority relation, depicted on the right-hand side, is cyclic. Each column represents the strict preference relation of a voter, \eg $a\mathrel{P_1} b \mathrel{P_1} c$.}
 \label{fig:ccycle}
\end{figure}

Inspired by earlier work by \citet{Sen69a}, \citet{Bord76a} factorized the set equality of the WARP condition into two inclusion conditions, namely $\alpha$ and $\beta^+$. Accordingly, WARP is equivalent to the conjunction of $\alpha$, an inclusion \emph{contraction} condition, and $\beta^+$, an inclusion \emph{expansion} condition.

\begin{definition}
An SCF satisfies $\alpha$ if $B\cap f(A,R)\subseteq f(B,R)$ for all $A,B\in \fone(X)$ such that $B\subseteq A$ and $B\cap f(A,R)\ne \emptyset$.
\end{definition}

Intuitively, $\alpha$ means that if an alternative is chosen from some set of alternatives, then it will also be chosen from all subsets in which it is contained. Clearly, the corresponding statement in terms of set expansion (a chosen alternative will be chosen in all supersets) would be unduly strong as \emph{every} alternative is chosen in some singleton set. The intuition behind $\beta^+$ is that if alternative $a$ is chosen from some set that contains another alternative $b$, then it will also be chosen in all supersets, which contain $a$ and in which $b$ is chosen.

\begin{definition}
An SCF satisfies $\beta^+$ if $f(B,R) \subseteq B\cap f(A,R)$ for all $A,B\in \fone(X)$ such that $B\subseteq A$ and $B\cap f(A,R)\ne \emptyset$.
\end{definition}

Conditions $\alpha$ and $\beta^+$ are known as the strongest contraction and expansion consistency condition, respectively. It turns out that the mildest dose of contraction consistency gives rise to impossibility results that retain Arrow's spirit \citep{Sen77a}. Expansion consistency conditions, on the other hand, are much less restrictive. In fact, under the given conditions, $\beta^+$, the strongest expansion condition, and minimality characterize the top cycle $\tc$ \citep{Bord76a}. It therefore appears advisable to discard $\alpha$ and characterize SCFs using weakened versions of $\beta^+$. This has, for example, been achieved for the uncovered set, which is characterized by a weakening of $\beta^+$ called $\gamma$ and minimality \citep{Moul86a}.

Anonymity, neutrality, and positive responsiveness are elementary, undisputed conditions that are satisfied by all common SCFs. This leaves two avenues to generalize the correspondence above. First, when the number of agents is not odd or their preferences are not linear, the pairwise majority relation may not be complete; it can be represented by an oriented graph or so-called \emph{incomplete tournament}. While the definitions of many tournament solutions have been extended to this generalized setting, no uncontroversial extensions are known for some tournament solutions (see \charef{sec:conclusion}). Secondly, the pairwise choice condition is often replaced with the \emph{Condorcet condition}, which states that the Condorcet winner has to be the unique choice whenever it exists \citep{Cond85a}. This class includes all common tournament solutions as well as some SCFs that take into account more information than just pairwise comparisons, such as Kemeny's and Dodgson's SCFs \citep[see, \eg][]{Fish77a}.

\subsubsection{Variable Electorate}

In the previous section, we pointed out how to circumvent Arrow's impossibility by substituting the prohibitively strong WARP condition with weaker conditions. The essence of all these properties is to impose restrictions on SCFs based on changes in the set of feasible alternatives. Alternatively, one can focus on changes in the set of voters. A very natural consistency property with respect to the electorate was suggested independently by \citet{Smit73a} and \citet{Youn74a}. It states that all alternatives that are chosen simultaneously by two disjoint sets of voters should be precisely the alternatives chosen by the union of both sets of voters.\todo{is this definition formally correct? $N$ is fixed!}
\begin{definition}
An SCF $f$ is \emph{consistent} if $f(A,R\cup R')=f(A,R)\cap f(A,R')$ for all disjoint sets $N$ and $N'$, $R\in \mathcal{R}^N$, $R'\in\mathcal{R}^{N'}$, and $A\in \fone(X)$ such that $f(A,R)\cap f(A,R')\ne \emptyset$.
\end{definition}

\citet{Smit73a} and \citet{Youn75a}, again independently, showed that consistency characterizes so-called scoring rules. For the remainder of this section, we assume that individual preference relations are linear. 
Let a \emph{score function} $s:\mathbb{N}\times\mathbb{N}\rightarrow \mathbb{R}$ be a function that for a given number of alternatives $m$ and rank $r$ yields the score $s(r,m)$ of the $r$th-ranked alternative in a linear order. Every score function $s$ yields an scf $f_s$ by letting 
\[f_s(A,R)=\{ a\in A \mid S(a)\ge S(b) \text{ for all }b\in A\}\] 
where $S(a)=\sum_{i\in N} s(|\{b\in A\mid b \mathrel{P_i} a\}|+1,|A|)$.
$f_s$ will be called a \emph{simple scoring rule}.
Simple scoring rules are by far the most used voting rules in practice. Common simple scoring rules are \emph{plurality} ($s(r,m)$ equals $1$ if $r=1$ and $0$ otherwise), \emph{Borda} ($s(r,m)=m-r$), and \emph{anti-plurality} ($s(r,m)$ equals $0$ if $r=m$ and $1$ otherwise).

For an anonymous SCF $g$ and a simple scoring rule $f_s$, the composition of $f_s$ and $g$ is defined as
\[(f_s\circ g) (A,R)=\{ a\in g(A,R) \mid S(a)\ge S(b) \text { for all }b\in g(A,R)\}\]
where $S(a)$ is defined as above. That is, the scores of $f_s$ are used to break the ties of $g$.

\begin{definition}
An SCF $f$ is a \emph{scoring rule} if there is a sequence of simple scoring rules $s_1,\dots,s_k$ such that $f=f_{s_1}\circ \dots \circ f_{s_k}$.
\end{definition}

\begin{theorem}[\citealp{Smit73a,Youn75a}]
An SCF is a scoring rule if and only if it satisfies anonymity, neutrality, and consistency.
\end{theorem}

One of the earliest advocates of scoring rules was the Chevalier de Borda, whose dispute with the Marquis de Condorcet in the French Royal Academy of Sciences is often regarded the birthplace of social choice theory \citep[see, \eg][]{Blac58a,Youn88a,Youn95a}. Interestingly, the disagreement between Borda and Condorcet even prevails in contemporary social choice theory and is manifested in the following theorem.
\begin{theorem}[\citealp{YoLe78a}]
There is no SCF that satisfies the Condorcet condition and consistency.
\end{theorem}

\citet{Lasl96a} further deepens the dichotomy between consistency with respect to the set of alternatives and the electorate by showing that no rank-based Paretian SCF (and thus no scoring rule) satisfies composition-consistency. Moreover, \citet{Lasl00a} contrasts the characterization of Borda's rule in terms of a variable electorate \citep{Smit73a,Youn74a} with a characterization of a generalization of the bipartisan set in terms of a variable set of alternatives.


\subsection{Adversarial Decision-Making}\label{sec:adversarial}

One of the oldest endeavors in \emph{non-cooperative game theory} is to investigate the optimal course of action in strictly competitive situations with two agents. The characteristic difficulty of such situations is that the optimality of an action depends on the action chosen by the other player.
Game theory has numerous applications in economics, biology, philosophy, and computer science.


Let $X$ be a universal set of actions. A \emph{two-player zero-sum game} $\Gamma=(A_1,A_2,u)$ is a tuple consisting of a finite non-empty subset~$A_1\in \fone(X)$ of feasible actions of player~$1$, a finite non-empty set~$A_2\in \fone(X)$ of feasible actions of player~$2$, and a utility function $u:X\times X\rightarrow\mathbb{R}$. 
Both players are assumed to simultaneously choose one of their actions, $a_1\in A_1$ and $a_2\in A_2$, in order to maximize their utility, which is given by  $u(a_1,a_2)$ in the case of player~$1$ and $-u(a_1,a_2)$ in the case of player~$2$. 
It is convenient to represent a zero-sum game as a matrix $U=(U_{i,j})_{i\in A_1, j\in A_2}$ with $U_{i,j}=u(i,j)$. 
The set of all utility functions on $X$ will be denoted by $\mathcal{U}(X)$.

In this section, we will be concerned with \emph{set-valued solution concepts}, \ie solution concepts that identify subsets of preferable actions for both players.
\begin{definition}
An \emph{adversarial game-theoretic solution concept (AGS)} is a function $f: \fone(X) \times \fone(X) \times \mathcal{U}(X) \rightarrow \fone(X)\times \fone(X)$ such that, for every zero-sum game $(A_1,A_2,u)$, $B_1\subseteq A_1$ and $B_2\subseteq A_2$ where $(B_1,B_2)=f(A_1,A_2,u)$.
\end{definition}

A zero-sum game $\Gamma=(A_1,A_2,u)$ is called \emph{symmetric} if $A_1=A_2$ and $u(a,b)=-u(b,a)$ for all $a,b\in A_1$.  In other words, $\Gamma$ is symmetric if and only if $A_1=A_2$ and $U$ is \emph{skew symmetric}, \ie $U^T=-U$.
In the following, we will restrict our attention to so-called tournament games, a subclass of symmetric zero-sum games that is characterized by the fact that players may only win, lose, or draw and that both players get the same payoff if and only if they play the same action \citep[see, \eg][]{LLL93b,FiRy95a}.

\begin{definition}
A symmetric zero-sum game $\Gamma=(A,A,u)$ is a \emph{tournament game} if the domain of $u$ is $\{-1,0,1\}$ and, for all $a,b\in A$, $u(a,b)=0$ if and only if $a=b$. 
\end{definition}
Even though tournament games appear to be rather restricted, they are rich enough to model most of the characteristic peculiarities of zero-sum games.

Two games $\Gamma=(A_1,A_2,u)$ and $\Gamma'=(A_1',A_2',u')$ are isomorphic if there is a pair of bijective mappings $\pi_1:A_1\rightarrow A_1'$ and $\pi_2:A_2\rightarrow A_2'$ such that $u(\pi_1(a_1),\pi_2(a_2))=u'(a_1,a_2)$ for all $a_1\in A_1$ and $a_2\in A_2$. It is very natural to assume that an AGS is independent of infeasible actions and stable with respect to isomorphisms. We will call such an AGS \emph{independent}.
\begin{definition}
An AGS $f$ is \emph{independent} if 
\begin{enumerate}[label=\textit{(\roman*)}]
\item $f(\Gamma)=f(\Gamma')$ for all games $\Gamma=(A_1,A_2,u)$ and $\Gamma'=(A_1,A_2,u')$ such that $u(a_1,a_2)=u'(a_1,a_2)$ for all $a_1\in A_1$, $a_2\in A_2$, and
\item $f((\pi_1(A_1),(\pi_2(A_2),u'))=(\pi_1(f(\Gamma),\pi_2(f(\Gamma)))$ for all games $\Gamma=(A_1,A_2,u)$ and $\Gamma'=(A_1',A_2',u')$, and isomorphisms $\pi_1:A_1\rightarrow A_1'$ and $\pi_2:A_2\rightarrow A_2'$ of $\Gamma$ and $\Gamma'$.
\end{enumerate}
\end{definition}

While not as natural as the previous condition, there are reasons to assume that an AGS should give identical recommendations to both players in a symmetric game. 
\begin{definition}
An AGS $f$ is \emph{symmetric} if for all symmetric games $\Gamma$, $B_1=B_2$ where $(B_1,B_2)=f(\Gamma)$. 
\end{definition}

There is a straightforward correspondence between tournament games and tournaments as the utility matrix of any tournament game may be interpreted as the adjacency matrix of a tournament graph. For a given tournament game $(A,A,u)$, a tournament $(A,{\succ})$ can be constructed by letting, for all $a,b\in A$, $a\succ b$ if $u(a,b)=1$ and $b\succ a$ if $u(a,b)=-1$. We thus have a direct relationship between tournament solutions and symmetric and independent AGSs in the case of tournament games.

\begin{quote}
For tournament games, every symmetric and independent AGS corresponds to a tournament solution and vice versa.
\end{quote}

As an example, \figref{fig:rps} shows the well-known tournament game Rock-Paper-Scissors and the corresponding tournament, a 3-cycle.

\begin{figure}[htb]
  \centering\figuresize
  $\begin{array}{rcrrrc}
& &  R & P & S & \\[1ex]
 R & \multirow{3}{1ex}{$\left(\begin{array}{r}\phantom{1}\\\phantom{1}\\\phantom{1}\end{array}\right.$} & 0  & -1 & +1 & \multirow{3}{0ex}{$\left)\begin{array}{l}\phantom{1}\\\phantom{1}\\\phantom{1}\end{array}\right.$}\\
 P & & +1 &  0 & -1 & \\
 S & & -1 & +1 &  0 & 
  \end{array}$
  \qquad\qquad  
  \def\nd{0.88em} 
  \def\ra{0.7em} 
  \begin{tikzpicture}[baseline=(current bounding box.center),node distance=\nd,vertex/.style={circle,draw,minimum size=2*\ra,inner sep=0pt}]
	\draw
    +(90:\nd+\ra) node[vertex] (r) {$R$} 
    +(210:\nd+\ra) node[vertex] (p)  {$P$}
    +(330:\nd+\ra) node[vertex] (s)  {$S$};
	\draw [-latex] (r) to (s);
	\draw [-latex] (s) to (p);
	\draw [-latex] (p) to (r);
 \end{tikzpicture}
 \caption{Tournament game Rock-Paper-Scissors and its corresponding tournament graph}
 \label{fig:rps}
\end{figure}

A \emph{best response} is an action that maximizes a player's utility for a given action of the opponent. Action $a$ is a \emph{better response} than action $b$ for some action of the opponent, if $a$ yields more utility than $b$.
In tournament games, where actions correspond to alternatives, the set of best responses for a given action simply consists of the dominators of the corresponding alternative. An action that is the unique best response to all actions of the other player is called a \emph{dominant action}. Clearly, such an action should be played without any reservations. In tournament games, any such action corresponds to a Condorcet winner of the corresponding tournament. Interestingly, many other solution concepts that have been developed independently in game theory are equivalent to certain tournament solutions within the restricted class of tournament games (see \tabref{tbl:gtconcepts}).

For example, action $a$ is \emph{dominated} if 
there exists another action $b$ such that, for all actions of the opponent, $b$ is always a better response then $a$. In tournament games, an action is dominated if it is never a best response, which is only possible if it is the Condorcet loser of the corresponding tournament. It may be the case that removing a dominated action from consideration renders another action dominated. This enables the definition of an \emph{iterative process} in which dominated actions are deleted one after another. The resulting set of actions is the set of iterated undominated or \emph{rationalizable} actions. 

Action $a$ is \emph{weakly dominated} if there exists another action $b$ such that, for all actions of the opponent, $a$ is never a better response than $b$ and $b$ is a better response at least once. The set of weakly undominated actions corresponds to the uncovered set. As in the case of strict dominance, one can define an iterative process of deleting weakly dominated actions. In general games, the resulting set of actions depends on the order of eliminations. In tournament games, however, this process invariably leads to the same set of actions, which corresponds to the iterated uncovered set. 

Another solution concept in game theory is that of a \emph{CURB (``closed under rational behavior'') set} \citep{BaWe91a}. A CURB set if a set of actions for each player such that each set contains all best responses to actions within the other player's set.\footnote{Originally, CURB sets are defined as sets of mixed, \ie randomized, strategies. However, both definitions are equivalent in tournament games.} Typically, one is interested in inclusion-minimal CURB sets. Every tournament game contains a unique minimal CURB set which coincides with the top cycle of the tournament. 

\citet{Shap64a} proposed a similar set-valued solution concept.
A weak generalized saddle point (WGSP) is a set of actions for each player such that every action not contained in these sets is weakly dominated when restricting the opponent's actions to those included in his set. An inclusion-minimal WGSP is called a weak saddle and is equivalent to the minimal covering set of the corresponding tournament. This equivalence is particularly astounding as it was discovered eight years after the minimal covering set was proposed and 43 years after the weak saddle was first mentioned \citep{DuLe96a}. 

Finally, the bipartisan set is a tournament solution whose definition directly refers to the corresponding tournament game. A (mixed) strategy is a probability distribution over the actions of one player. A pair of strategies is a \emph{Nash equilibrium} if neither of the two players can increase his expected utility by deviating from his strategy \citep{Nash51a}. Thus, a Nash equilibrium is a mutual mixed best response. Tournament games contain a unique Nash equilibrium and the bipartisan set is defined as the set of all alternatives that correspond to actions that are played with positive probability in the Nash equilibrium.\footnote{Mutual best response \emph{actions} constitute a \emph{pure} Nash equilibrium.  In tournament games, a pure Nash equilibrium exists if and only if the tournament contains a Condorcet winner.}

\begin{table}[htb]
\centering\figuresize
\begin{tabular}{ll}
\toprule
Tournament solution & Game-theoretic concept\\
\midrule
Condorcet winner & Dominant action/pure Nash equilibrium\\
Condorcet non-losers ($\cnl$) & Undominated actions\\ 
Uncovered set ($\uc$) & Weakly undominated actions\\
Iterated Condorcet non-losers ($\cnl^\infty$) & Rationalizability\\
Iterated uncovered set ($\uc^\infty$) & Iterated weakly undominated actions\\
Top cycle ($\tc$) & Minimal CURB set\\ 
Minimal covering set ($\mc$) & Weak saddle\\ 
Bipartisan set ($\bp$) & Support of Nash equilibrium\\ 
\bottomrule
\end{tabular}
\caption{Tournament solutions and their game-theoretic counterparts}
\label{tbl:gtconcepts}
\end{table}

\subsection{Coalitional Decision-Making}\label{sec:coalitional}
\todo{use hedonic game example from Barbera and Gerber (2007)!}

\emph{Coalitional game theory} (or \emph{cooperative game theory}) studies strategic settings in which players can make binding commitments, form coalitions, and thus correlate their actions. Here, we will consider what is known as the case of \emph{non-transferable utility (NTU)}, \ie there is no possibility of side-payments between players.

A \emph{finite NTU game} is a tuple $(N,H,V)$, where $N=\{1,\dots,n\}$ is a set of players, $H\in\fone(\mathbb R^N)$ is a set of outcomes, and~$V:\fone(N)\rightarrow \fone(H)$ is a characteristic function.
Each outcome denotes how much utility each agent derives from this particular outcome. The \emph{characteristic function} describes the coalitional effectivity, \ie it yields which outcomes each coalition can enforce to come about. An important question is which outcomes are to be expected in a given NTU game when players are assumed to maximize their utility. Differing answers to this question have been given in the form of \emph{solution concepts} that map NTU games to sets of outcomes.  Some of these concepts (\eg the core or von Neumann-Morgenstern stable sets) are defined on the basis of a binary dominance relation that is typically defined in terms of coalitional effectivity and individual preferences. 

\begin{definition}
Let $(N,H,V)$ be a finite NTU game and $a,b\in H$ two outcomes. $a$ \emph{dominates} $b$ if there exists a coalition $C\in \fone(N)$ such that $C$ is effective for $a$, \ie $a\in V(C)$, and all members of $C$ strictly prefer $a$ over $b$, \ie $a(i)> b(i)$ for all $i\in C$. 
\end{definition}
Obviously this dominance relation is irreflexive since no outcome dominates itself. However, it may be the case that two outcomes dominate each other (via different coalitions). In other words, the dominance relation does not have to be asymmetric and thus differs from the dominance relations we have considered so far. Nevertheless, the specific structural properties of the dominance relation are of great importance when reasoning about coalitional games. It turns out that \emph{every} irreflexive relation on a finite set of alternatives can be obtained as the dominance relation of some finite coalitional NTU game.

\begin{theorem}[\citealp{BrHa09a}]\label{thm:ntu}
Let~$R$ be an irreflexive relation on a {finite} set of outcomes. Then, $R$ is induced as the dominance relation of some finite NTU game.
\end{theorem}
This result is very similar in spirit to that of \citet{McGa53a} who has shown that any asymmetric dominance relation can be obtained via pairwise majority voting (see \secref{sec:varalt}). In general, \emph{cooperative majority voting} can be seen as the special case of finite NTU games in which majorities are universally effective and all other coalitions are impotent \citep[see \eg][]{Schw86a}.

By now it has become obvious that all solution concepts studied in this \doc, be their roots in social choice theory, non-cooperative game theory, or in cooperative game theory, have to deal with what is essentially the same problem: to come to grips with a possibly intransitive dominance relation. Each of them incorporates a different intuition and approaches the issue from a different angle. In order to apply tournament solutions as solution concepts for finite NTU games, however, we need to bridge the gap between asymmetric and complete dominance relations on the one hand, and irreflexive dominance relations on the other. This can be achieved either by extending the definitions of tournament solutions to irreflexive relations or by identifying reasonable subclasses of finite NTU games (other than cooperative majority voting) that yield asymmetric and complete dominance relations. Both approaches are left as future work.
\todo{strong simple games?}

}

\begin{savequote}
\sffamily
There is every reason to expect that the various social sciences will serve as incentives for the development of great new branches of mathematics and that some day the theoretical social scientist will have to know more mathematics than the physicist needs to know today.
\qauthor{J.~G.~Kemeny, 1959}
\end{savequote}

\extra{\section{Conclusions and Open Problems}}
\extrapaper{\section{Conclusion}}
\label{sec:conclusion}

We proposed a unifying treatment of tournament solutions based on maximal qualified subsets and minimal stable sets. 
Given the results of \charef{sec:minstablesets} and \charef{sec:teq}, a central role in the theory of tournament solutions may be ascribed to Conjecture~\ref{con:teq}, a statement of considerable mathematical depth. Conjecture~\ref{con:teq} has a number of appealing consequences on minimal stable sets, some of which have been proved already. 

\begin{enumerate}[label=\textit{(\roman*)}]
\item Every tournament $T$ admits a unique \emph{minimal dominant set} $\tc(T)$ (as shown by~\citealp{Good71a}).
$\tc$ satisfies all basic properties and is the finest solution concept satisfying \ssp and $\cnl$-exclusivity.
\item Every tournament $T$ admits a unique \emph{minimal covering set} $\mc(T)$ (as shown by~\citealp{Dutt88a}). 
$\mc$ satisfies all basic properties and is the finest solution concept satisfying \ssp and $\uc$-exclusivity.
\item Every tournament $T$ admits a unique \emph{minimal extending set} $\me(T)$ (open problem).
$\me$ satisfies all basic properties and is the finest solution concept satisfying \ssp and $\ba$-exclusivity.
\item Every tournament $T$ admits a unique \emph{minimal $\teq$-retentive set} $\teq(T)$ (open problem).
$\teq$ satisfies all basic properties and is the finest solution concept satisfying retentiveness and the finest solution concept $S$ such that $S$ satisfies \ssp and, for all tournaments~$T=(A,\succ)$, $S(A)=A\setminus \{a\}$ only if $a\not\in S(\dom(b))$ for every $b\in A$. 
\item The following inclusion relationships hold: $\teq\subseteq\me\subseteq\mc\subseteq\tc$ and $\me\subseteq \ba$.\footnote{A consequence of these inclusions is that deciding whether an alternative is contained in the minimal extending set of a tournament is NP-hard. \extrapaper{This follows from a proof by \citet{BFHM09a}, which establishes hardness of all solution concepts that are sandwiched between $\ba$ and $\teq$.}}
\end{enumerate}

Conjecture~\ref{con:me} is a weaker version of Conjecture~\ref{con:teq}, which implies all of the above statements except those that involve $\teq$.


Table~\ref{tbl:comparison} and Figure~\ref{fig:sets} summarize the properties and set-theoretic relationships of the considered tournament solutions, respectively.

\extra{
\pagebreak
Many challenging open problems remain. Among them are the following:

\paragraph{Stability and retentiveness.}
Prove or disprove that no tournament admits two disjoint $S_{\mathcal{M}^*_4}$-stable sets.
Prove or disprove that no tournament admits two disjoint extending sets (\conref{con:me}).
Prove or disprove that no tournament admits two disjoint $\teq$-retentive sets (\conref{con:teq}). Investigate $S^\#_{\mathcal{M}^*}$ and $\mcs{S}^\#_{\mathcal{M}^*}$. Study $\mr{S}$ for solution concepts $S$ other than $\teq$.

\paragraph{Generalization of tournament solutions.}
The definitions of most of the concepts considered in this \doc have been generalized to incomplete tournaments \citep{DuLa99a,PeSu99a}. \citet{BrFi08b} and \citet{BFH09b} investigate whether and how these generalizations affect the computational complexity. For some concepts ($\ba$, $\me$, and $\teq$) no uncontroversial generalization is known. Ideally one would want to extend the definition to arbitrary binary relations.

\paragraph{Computational aspects.}
Is there a simpler method for computing a subset of $\mc$ than solving a linear feasibility problem?
Is there a purely combinatorial algorithm for computing $\mc$ (one that does not rely on linear feasibility)? How hard is it to compute $S_{\mathcal{M}_k}$? Is it possible to compute a subset of $\me$ or $\teq$ in polynomial time? Pinpoint the complexity of $\me$ and $\teq$ (the best upper bounds are $\Pi_2^p$ and PSPACE, respectively). Is deciding membership in $\mc$ (or $\uc^\infty$) \tczero-hard. Is deciding membership in $\bp$ or $\mc$ P-complete in tournaments? Is there a linear-time algorithm for computing $\uc$? 

\bigskip

Apart from these technical issues, it remains to be seen which other application areas (besides the ones already mentioned in \charef{sec:applications}) will be found for tournament solutions. It is to be expected that the problem of identifying the ``best'' elements according to some binary relation arises in numerous contexts and various fields. 

}

\begin{table}[htb]
\begin{minipage}{\textwidth}
\centering\footnotesize
\begin{tabular}{lllcccccc}
\toprule
 \multicolumn{2}{l}{Solution Concept} & Origin & \mon & \iua & \wsp & \ssp & \com & \irr \\
\midrule
$S_{\mathcal{M}_2}$ & $(\cnl)$ & & \checkmark & \checkmark & \checkmark & -- & -- & -- \\
$S_{\mathcal{M}}$ & $(\uc)$ & \citet{Fish77a,Mill80a}  & \checkmark & -- & \checkmark & -- & \checkmark & -- \\
$S_{\mathcal{M}^*}$ & $(\ba)$ & \citet{Bank85a}  & \checkmark & -- & \checkmark & -- & \checkmark & \checkmark \\
$\ms{S}_{\mathcal{M}_2}$ & $(\tc)$ & \citet{Good71a,Smit73a}  & \checkmark & \checkmark & \checkmark & \checkmark & -- & -- \\
$\ms{S}_{\mathcal{M}}$ & $(\mc)$ & \citet{Dutt88a}  & \checkmark & \checkmark & \checkmark & \checkmark & \checkmark & -- \\
$\ms{S}_{\mathcal{M}^*}$ & $(\me)$ & 
  & \checkmark$^\text{\it a}$ & \checkmark$^\text{\it a}$ & \checkmark\footnote{This statement relies on Conjecture~\ref{con:me}.} & \checkmark$^\text{\it a}$ & \checkmark & \checkmark \\
$\ms{\teq}$ & $(\teq)$ & \citet{Schw90a} & \checkmark\footnote{This statement relies on Conjecture~\ref{con:teq}.} & \checkmark$^\text{\it b}$ & \checkmark$^\text{\it b}$ & \checkmark$^\text{\it b}$ & \checkmark & \checkmark \\
\midrule
$S^\#_{\mathcal{M}}$ & $(\co)$ & \citet{Cope51a}  & \checkmark & -- & -- & -- & -- & -- \\
$\mcs{S}^\#_{\mathcal{M}}$ & $(\bp)$ & \citet{LLL93b}  & \checkmark & \checkmark & \checkmark & \checkmark & \checkmark & -- \\
\bottomrule
\end{tabular}
\end{minipage}
\caption{Properties of solution concepts
(\mon: monotonicity,
\iua: independence of unchosen alternatives,
\wsp: weak superset property,
\ssp: strong superset property,
\com: composition-consistency,
\irr: irregularity).
See \citet{Lasl97a} for all results not shown in this paper.
}
\label{tbl:comparison}
\end{table}

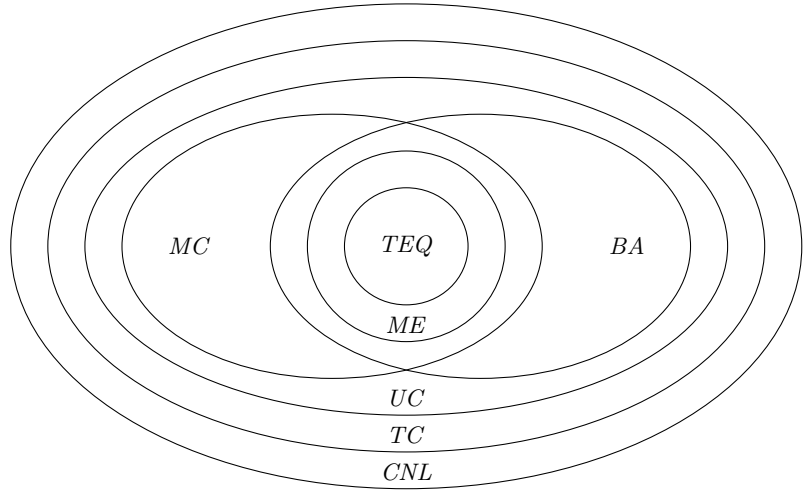
\begin{figure}[htb]
  \centering\footnotesize
  \newcommand{\sw}{32em} 
  \newcommand{\sd}{3em} 
  \newcommand{\ra}{0.618} 
\begin{tikzpicture}
\draw (0,0) node[ellipse, draw, minimum width=\sw, minimum height=\ra*\sw] (CNL) {};
\draw (0,0) node[ellipse, draw, minimum width=\sw-\sd, minimum height=\ra*\sw-\sd] (TC) {};
\draw (0,0) node[ellipse, draw, minimum width=\sw-2*\sd, minimum height=\ra*\sw-2*\sd] (UC) {};
\draw (-\sd,0) node[ellipse, draw, minimum width=\sw-5*\sd, minimum height=\ra*\sw-3*\sd] (MC) {};
\draw (\sd,0) node[ellipse, draw, minimum width=\sw-5*\sd, minimum height=\ra*\sw-3*\sd] (BA) {};
\draw (0,0) node[ellipse, draw, minimum width=\sw-8*\sd, minimum height=\ra*\sw-4*\sd] (ME) {};
\draw (0,0) node[ellipse, draw, minimum width=\sw-9*\sd, minimum height=\ra*\sw-5*\sd] (TEQ) {$\teq$};

\draw[text=black] (0,-0.5*\ra*\sw) node[anchor=south](cnl){$\cnl$};
\draw[text=black] (0,-0.5*\ra*\sw+0.5*\sd) node[anchor=south](tc){$\tc$};
\draw[text=black] (0,-0.5*\ra*\sw+\sd) node[anchor=south](uc){$\uc$};
\draw[text=black] (-0.5*\sw+2*\sd,0) node[anchor=west](mc){$\mc$};
\draw[text=black] (0.5*\sw-2*\sd,0) node[anchor=east](ba){$\ba$};
\draw[text=black] (0,-0.5*\ra*\sw+2*\sd) node[anchor=south](me){$\me$};
\end{tikzpicture}
 \caption{Set-theoretic relationships between qualitative tournament solutions.
$\ba$ and $\mc$ are not included in each other, but they always intersect. 
The inclusion of $\teq$ in $\me$ relies on Conjecture~\ref{con:teq} and that of $\me$ in $\mc$ on Conjecture~\ref{con:me} (which is implied by Conjecture~\ref{con:teq}). $\co$ is contained in $\uc$ but may be disjoint from $\mc$ and $\ba$. The exact location of $\bp$ in this diagram is unknown ($\bp$ is contained in $\mc$ and is a superset or subset of  $\teq$ in all known instances \citep{Lasl97a}).}
 \label{fig:sets}
\end{figure}

\section*{Acknowledgments}

I am indebted to Paul Harrenstein for countless valuable discussions.
Furthermore, I thank Haris Aziz, Markus Brill, and Felix Fischer for providing helpful feedback on drafts of this \doc.
This material is based on work supported by the Deutsche Forschungsgemeinschaft under grants BR~2312/3-2, BR~2312/3-3, and BR~2312/7-1.

\extra{
The chapter quotes are taken from the following publications. Chapter~\ref{sec:intro}: \citet*[page~453]{Bord76a}, Chapter~\ref{sec:prelim}: \citet*[page~169]{ReBe78a}, Chapter~\ref{sec:maxsubsets}: \citet*[page~577]{Keme59a}, Chapter~\ref{sec:minstablesets}: \citet*[page~42, `as' has been replaced with `that' in order to rectify what appears to be a typographical error]{vNM44a}, Chapter~\ref{sec:teq}: \citet*[page~1]{Schw86a}, Chapter~\ref{sec:quant}: \citet*[page~437]{Zerm29a}, Chapter~\ref{sec:computation}: \citet*[page~202]{vDam98a}, Chapter~\ref{sec:applications}: \citet*[page~2]{Land51a}, Chapter~\ref{sec:conclusion}: \citet*[page~578]{Keme59a}.
}

\extrapaper{Preliminary results of this paper were presented at the 9th International Meeting of the Society of Social Choice and Welfare (Montreal, June 2008).}

\bibliography{../pamas/abb,../pamas/brandt,../pamas/pamas}

\extra{\bigskip\bigskip\noindent Munich, \today}

\end{document}